\documentclass{article}
\usepackage{mysty_books}
\usepackage[backend=biber,style=alphabetic,citestyle=alphabetic,maxbibnames=5]{biblatex}
\usepackage[margin=1.2in]{geometry}

\bibliography{bibliography.bib}

\newcommand{\ty}{\tilde{y}}
\newcommand{\ta}{\tilde{a}}

\newcommand{\WidestPart}{\ensuremath{p_j^{e_j-1}\left(p_j-1-\Legendre{p}{p_j}\right)}}%
\newcommand{\FixedSize}[1]{\makebox[\widthof{\WidestPart}][l]{\ensuremath{#1}}}%

\DeclareMathOperator{\avg}{avg}

\counterwithin{equation}{section}
\renewcommand{\theequation}{(\arabic{section}.\arabic{equation})}

\title{\Large Murmurations of Hecke $L$-Functions of Imaginary Quadratic Fields}
\author{Zeyu Wang}
\date{March 2025}

\begin{document}
\maketitle


\begin{abstract}
We calculate the murmuration density for the family of Hecke $L$-functions of imaginary quadratic fields associated to non-trivial characters. This density exhibits a universality property like Zubrilina's density for the murmurations of holomorphic modular forms. We show all murmuration functions obtained by averaging over the family with a compactly supported smooth weight function has asymptotics compatible with the 1-level density conjecture of Katz and Sarnak. The novelty of the murmurations of this family of $L$-functions is its pronounced almost periodic feature, which allows one to describe this murmuration without averaging over primes, and which is non-existent or previously unnoticed for other families.
\end{abstract}

\section{Introduction}\label{chapter:introduction}
\subsection{Murmurations of Elliptic Curves and Other Families}\label{section:intro_1}
Recently, He, Lee, Oliver, and Pozdnyakov (see \cite{HLOP}) used machine learning techniques to discover an oscillation pattern in the averages of Frobenius traces of elliptic curves of fixed rank with conductor in a bounded interval (see Figure \ref{fig:HLOP}). This phenomenon is termed ``murmurations'' due to its visual similarity to bird flight patterns (see Figure \ref{fig:birds}). 

\begin{figure}[H]
    \centering
    \begin{subfigure}{0.48
    \textwidth}
        \includegraphics[width=\textwidth]{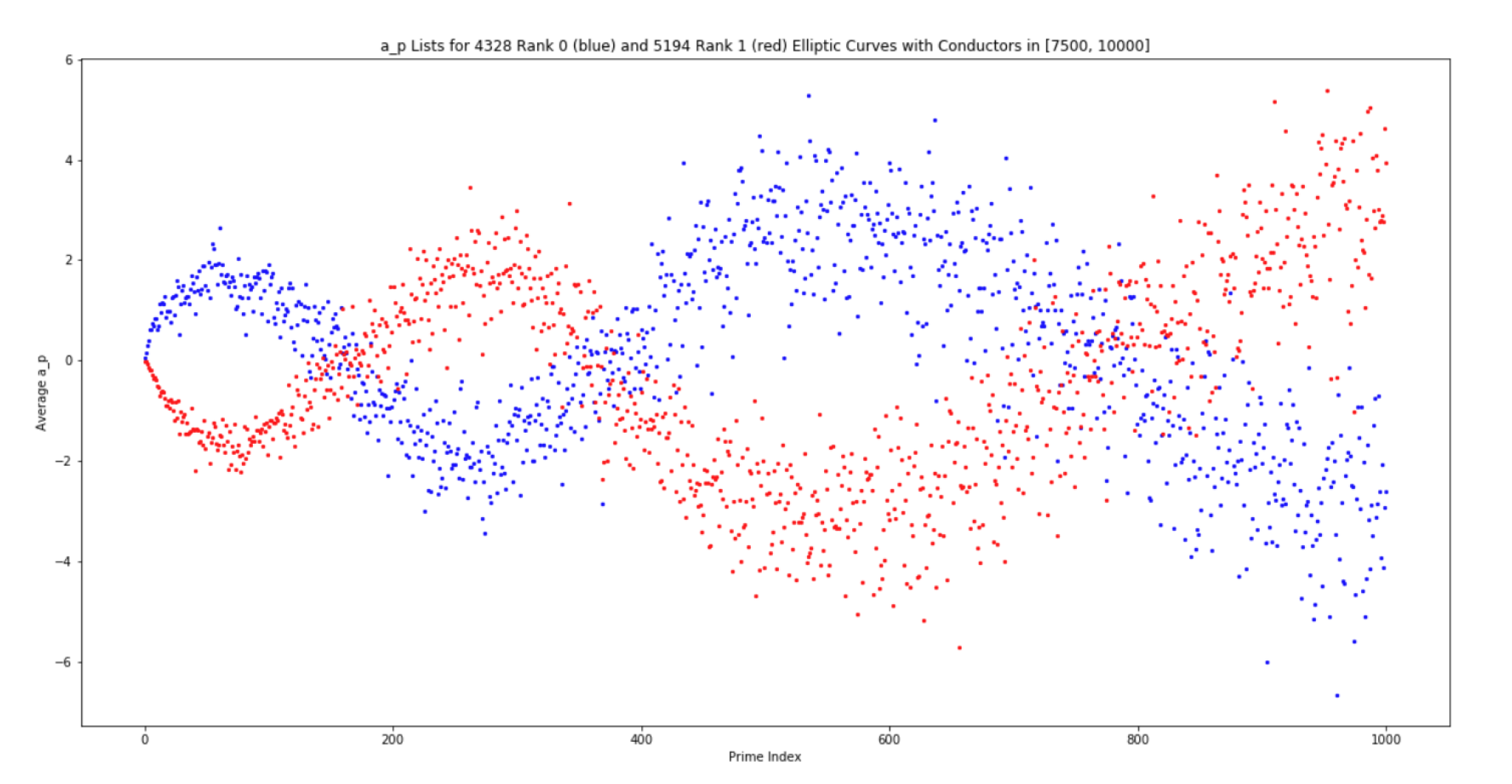}
        \caption{Frobenius trace averages of elliptic curves of fixed rank with conductor in a particular bounded interval.}
        \label{fig:HLOP}
    \end{subfigure}\hfill
    \begin{subfigure}{0.48\textwidth}
        \includegraphics[width=\textwidth]{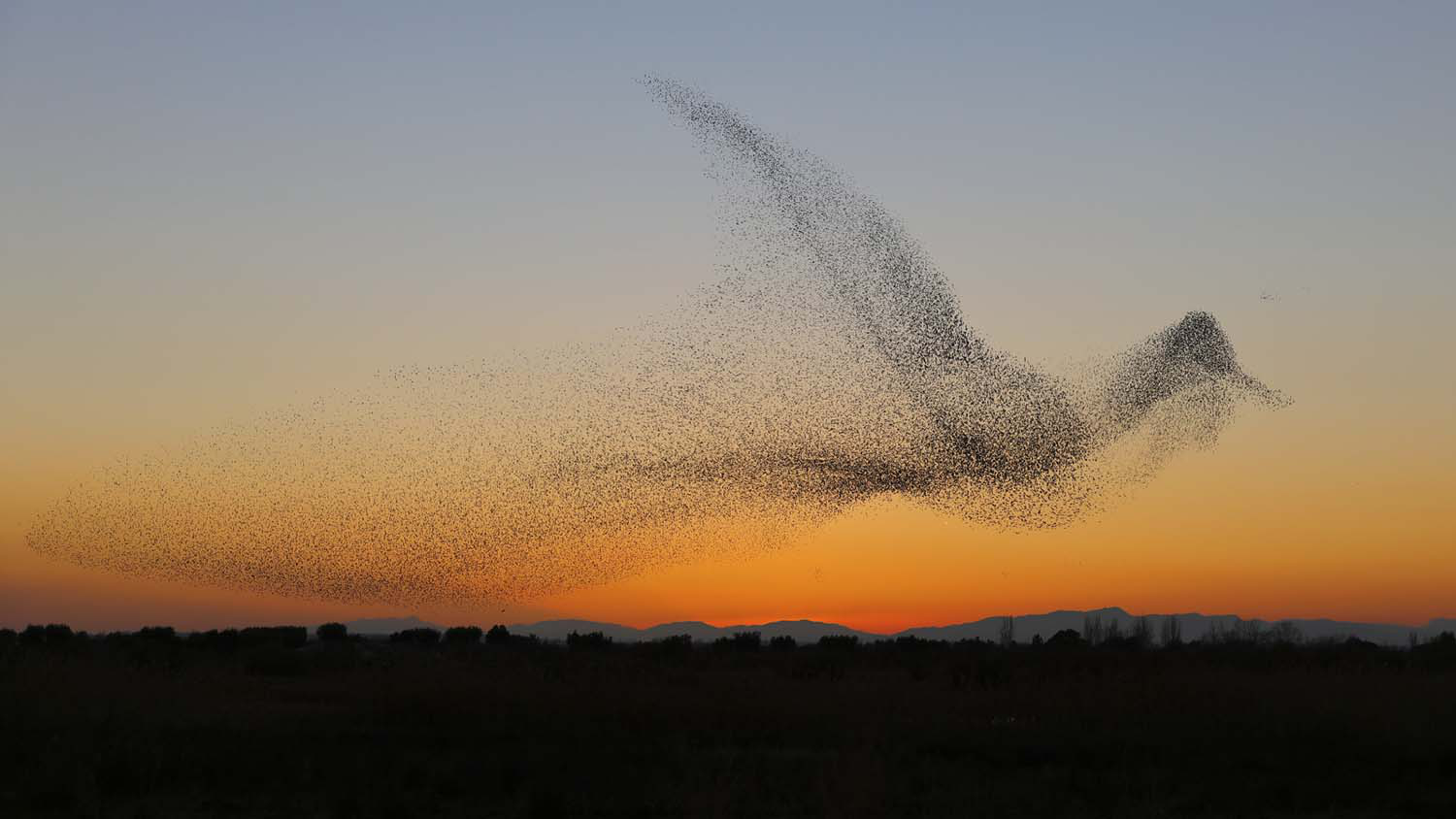}
        \caption{Birds.}
        \label{fig:birds}
    \end{subfigure}
    \caption{Oscillatory pattern observed in \cite{HLOP} and the murmurations of starlings.}
    \label{fig:enter-label}
\end{figure}

Sutherland and the authors (see \cite{HLOPS}) then discovered this phenomenon in more families of arithmetic $L$-functions, for example, those associated to weight-$k$ holomorphic cusp forms for $\Gamma_0(N)$ with conductor in $[N,cN]$ and a fixed root number $\varepsilon(f)$. The average of $a_f(p)$ over this family for a single prime $p\sim X$ converges to a continuous function of $p/N$ (see Figure \ref{fig:newforms}), and this has been proved by Zubrilina in \cite{Zubrilina}. Since then, murmuration phenomena in several other families of $L$-functions have been computed, such as that of Dirichlet $L$-functions in \cite{LOP_dirichlet}, that of modular forms in the weight aspect in \cite{BBLLD_modular_forms}, and that of weight-0 level-1 Maass forms in \cite{BLLDSHZ_Maass}.

\begin{figure}[H]
    \centering
    \includegraphics[width=\textwidth]{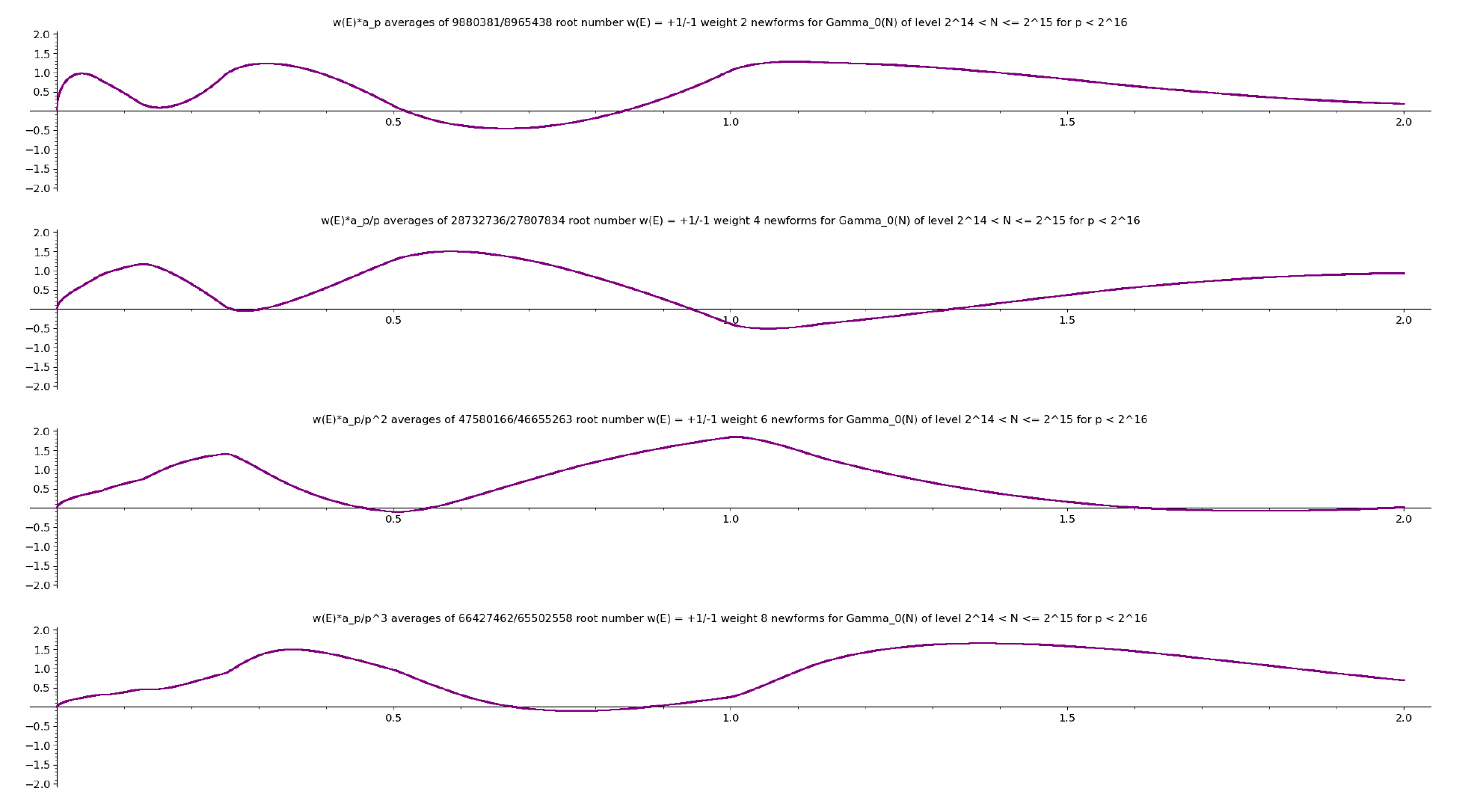}
    \caption{Frobenius trace averages of newforms of fixed weight with level in a particular bounded interval.}
    \label{fig:newforms}
\end{figure}

In this paper, we compute a similar murmuration phenomenon for another natural family of arithmetic $L$-functions, namely those associated to non-trivial Hecke characters on imaginary quadratic fields. The overall treatment and presentations of the main results will be most similar to those in \cite{Zubrilina}, but we will highlight several crucial differences.

\subsection{General Formulation of Murmurations}
Here we formulate the murmuration phenomenon in general, following Sarnak's letter \cite{SarnakLetter} to Sutherland and Zubrilina.

Let $\msF$ be a family of self-dual arithmetic $L$-functions ordered by conductor. For $f\in\msF$, let $N_f$ be the conductor of $f$, $\varepsilon(f)$ be the root number of $f$, and $a_f(p)$ be its Frobenius trace at $p$, normalized so that Ramanujan's conjecture states that $\abs{a_f(p)}\ll\sqrt{p}$. Note that the self-dual condition implies that $\varepsilon(f)=\pm1$ and $a_f(p)\in\RR$. We are interested in the weighted average of the $a_f(p)$ among those $f\in\msF$ with conductor $N_f$ multiplicatively close to $N$ (a ``vertical average''). Precisely, for any compactly supported smooth weight function $\Phi:(0,\infty)\to[0,\infty)$, we look for a function $M_\Phi(\xi)$ such that \begin{equation}
    G_\Phi(p,N):=\frac{\sum_{f\in\msF}a_f(p)\varepsilon(f)\Phi\left(\frac{N_f}{N}\right)}{\sum_{f\in\msF}\Phi\left(\frac{N_f}{N}\right)}=M_\Phi\left(\frac{p}{N}\right)+R_\Phi(p,N)
\end{equation} and such that the error term satisfies is negligible after averaging over primes in short intervals, i.e. \begin{equation}
    \frac{\sum_{p\in[P,P+H]}R_\Phi(p,N)}{\sum_{p\in[P,P+H]}1}=o(1)
\end{equation} where $H=o(N)$ is some parameter. We then say $M_\Phi(\xi)$ is the \textbf{murmuration function for $\msF$ with weight $\Phi$}. 

In other words, if we double sum over $f\in\msF$ and $p$, then the weighted average of $a_f(p)$ converges to $M_\Phi(p/N)$ as $p,N\to\infty$ with $\xi=p/N$ fixed. After switching the order of averaging, these double averages are related via explicit formulae to the 1-level densities of the zeros of $f\in\msF$, which are conjectured by Katz-Sarnak in \cite{Katz-Sarnak} to depend on the symmetry type of the family $\msF$. Based on the 1-level density conjecture, one expects that \begin{align}
    \Exp_{p\sim N^a}[G_\Phi(p,N)]&\to\begin{cases}
        0 &\text{if }a<1\\
        -\frac{1}{2} &\text{if }a>1
    \end{cases} &&(\text{if $\msF$ has $\Sp(\infty)$ symmetry})\label{eqn:Sp_1_level_density}\\
    \Exp_{p\sim N^a}[G_\Phi(p,N)]&\to\begin{cases}
        0 &\text{if }a<1\\
        \frac{1}{2} &\text{if }a>1
    \end{cases} &&(\text{if $\msF$ has $\O(\infty)$ symmetry}) \label{eqn:O_1_level_density}
\end{align} for any reasonable weight function $\Phi$ (see \cite{SarnakLetter}, p.5, Equation 7'). Thus we expect \begin{equation}\label{eqn:1_level_density}
    M_\Phi(0)=0\AND \lim_{\xi\to\infty}M_\Phi(\xi)=\begin{cases}-\frac{1}{2} &\text{if $\msF$ has $\Sp(\infty)$ symmetry,}\\\frac{1}{2} &\text{if $\msF$ has $\O(\infty)$ symmetry.}\end{cases}
\end{equation} Therefore, the murmuration function $M_\Phi(\xi)$ on $\RR_{\geq0}$ interpolates the phase transition of the 1-level density in the critical range $p\sim N$, corresponding to the discontinuity at $a=1$ of \ref{eqn:Sp_1_level_density} and \ref{eqn:O_1_level_density}.

The murmuration function depends on the choice of the weight function $\Phi$, and the more universal object is the ``murmuration density.'' In fact, if as $X,Y,p\to\infty$ with $Y=o(X)$ we have \begin{equation}\label{eqn:murmuration_density_defn}
    G(p,X,Y):=\frac{\sum_{\substack{f\in\msF\\N_f\in[X,X+Y]}}a_f(p)\varepsilon(f)}{\sum_{\substack{f\in\msF\\N_f\in[X,X+Y]}}1}=M\left(\frac{p}{X}\right)+R(p,X,Y)
\end{equation} with $R(p,X,Y)$ negligible after averaging over primes in intervals of length $H=o(X)$, then from $M$ we can recover $M_\Phi$ for any $\Phi$ by \begin{equation}\label{eqn:murmuration_function_and_density_relation}
    M_\Phi(\xi)=\frac{\int_0^\infty M(\xi/u)\Phi(u)u^\delta\frac{du}{u}}{\int_0^\infty\Phi(u)u^\delta\frac{du}{u}}.
\end{equation} In this case, we call $M$ the \textbf{murmuration density for $\msF$}, and this is the object of central interest. Here $\delta$ is the \textbf{conductor density} of the family $\msF$, which we assume satisfies \begin{equation}
    \#\{f\in\msF:N_f\leq N\}\asymp N^\delta.
\end{equation} For example, the conductor density is $\delta=2$ for the family of interest in \cite{Zubrilina}, and is $\delta=1$ for the families considered in \cite{LOP_dirichlet}, \cite{BBLLD_modular_forms}, and \cite{BLLDSHZ_Maass}. 

It is a natural question how much one needs to average over $p$ to obtain the murmuration density in \ref{eqn:murmuration_density_defn}, i.e. how short one may take $H=o(X)$ such that $$\Exp_{p\in[P,P+H]}R(p,X,Y)=o(1).$$ It is hypothesized in \cite{SarnakLetter} that one may take any $H=X^\gamma$ such that $\gamma+\delta>1$, where $\delta$ is again the conductor density. For example, in \cite{Zubrilina} one has $\delta=2$ and $\gamma=0$; that is, the family is sufficiently large such that no averaging over $p$ is required at all to obtain the murmuration density. 

In this paper, we compute the murmuration density for the family $\msF$ of $L$-functions associated to non-trivial Hecke characters on the class groups of imaginary quadratic fields $\QQ(\sqrt{-D})$ with $D>3,D\equiv3\Mod{4}$ and square-free. It is conjectured in \cite{Serre} that these dihedral forms make up the majority of holomorphic modular forms of weight 1. We also show that the murmuration function $M_\Phi$ for any compactly supported smooth weight function $\Phi$ has the asymptotic behavior in accordance with the 1-level density conjecture, i.e. $M_\Phi(\xi)\to-\frac{1}{2}$ as $\xi\to\infty$.

Interestingly, although the conductor density for our family $\msF$ is $\delta=\frac{3}{2}$, we still need to average a little over $p$ in order to recover the murmuration density as defined by \ref{eqn:murmuration_density_defn} in \cite{SarnakLetter}. However, without the averaging over $p$ we obtain an \textit{almost periodic} murmuration density, which is arguably more interesting and revealing in its own right.

\subsection{Properties of the Family $\msF$}
Here we collect some basic properties of the family $\msF$ (as well as the Hecke $L$-functions associated to trivial characters). 

Let $K=\QQ(\sqrt{-D})$ be an imaginary quadratic field of discriminant $-D\equiv1\Mod{4},D>3$, let $\psi$ be a (possibly trivial) character on $\Cl(K)$, and let $f(s)=L(s,\psi)$ be the corresponding Hecke $L$-function: \begin{equation}
    L(s,\psi)=\sum_{\mfa}\psi(\mfa)(N\mfa)^{-s}=\prod_\mfp\left(1-\psi(\mfp)(N\mfp)^{-s}\right)^{-1},
\end{equation} where $\mfa$ runs over all the integral ideals of $\mcO_K$ and $\mfp$ runs over all the prime ideals. Writing $f$ as a Dirichlet series $f(s)=\sum_{n=1}^\infty a_f(n)n^{-s}$, then the Frobenius traces are given by \begin{equation}\label{eqn:Frob_traces}
    a_f(p)=\sum_{\mfp,N\mfp=p}\psi(\mfp)=\begin{cases}
    0 &\Legendre{-D}{p}=-1\Leftrightarrow p\text{ inert},\\
    \psi(\mfp)+\psi(\mfpbarr) &\Legendre{-D}{p}=1\Leftrightarrow p\text{ split},p\mcO_K=\mfp\mfpbarr,\\
    \psi(\mfp) &\Legendre{-D}{p}=0\Leftrightarrow p\text{ ramified},p\mcO_K=\mfp^2.
\end{cases}
\end{equation} The completed function $$\Lambda(s,\psi)=\left(\frac{\sqrt{D}}{2\pi}\right)^s\Gamma(s)L(s,\psi)$$ is entire, with only possibly a simple pole at $s=1$ when $\psi=\ONE$ is trivial with residue $$\res_{s=1}\Lambda(s,\ONE)=\frac{\sqrt{D}}{2\pi}L(1,\chi_{-D})=\frac{1}{2}h(-D),$$ where $\chi_{-D}=\Legendre{-D}{\cdot}$ is the Kronecker symbol and $h(-D)=\#\Cl(K)$ is the class number. Moreover, these $L$-functions satisfy the functional equation $$\Lambda(s,\psi)=\Lambda(1-s,\psi),$$ so members $f\in\msF$ have root number $\varepsilon(f)=1$ always, which differs from the family considered in \cite{Zubrilina}.

The class number formula along with the bound $\abs{D}^{-\epsilon}\ll L(1,\chi_D)\ll\abs{D}^{\epsilon}$ implies that the family $\msF$ has conductor density $\delta=\frac{3}{2}$.\footnote{Throughout this paper, implied constants might depend on $\epsilon$ and we omit this from notation. The bound $L(1,\chi_D)\gg\abs{D}^{-\epsilon}$ is Siegel's Theorem, and $L(1,\chi_{D})\ll\log\abs{D}$ follows from P\'olya-Vinogradov.} This is a second distinction from \cite{Zubrilina}, where the conductor density is $\delta=2$.

Finally, the family $\msF$ is expected to have a symplectic $\Sp(\infty)$ symmetry (see \cite{Fouvry-Iwaniec}). Thus from \ref{eqn:1_level_density} we expect \begin{equation}\label{eqn:asymps_of_murmuration_function_F_only}
    M_\Phi(0)=0\AND\lim_{\xi\to\infty}M_\Phi(\xi)=-\frac{1}{2}
\end{equation} for the family $\msF$ for any reasonable weight function $\Phi$. This is a third distinction from \cite{Zubrilina}, where the family has an $\O(\infty)$ symmetry. 

\counterwithout{thm}{section}
\subsection{Results: Murmuration Density for the Family $\msF$}
Recall that $\msF$ is the family of $L$-functions associated to non-trivial Hecke characters on class groups of imaginary quadratic fields $\QQ(\sqrt{-D})$ with $D\equiv3\Mod{4},D>3$ and square-free, and $a_f(p)$ is the Frobenius trace of $f\in\msF$ at $p$ (with $\abs{a_f(p)}\leq2$, so $a_f(p)\sqrt{p}$ is the normalized trace).
\begin{thm}\label{thm:thm1}
Let $p,X,Y$ be parameters going to $\infty$ with $Y\sim X^{1-\delta_Y}$ and $p\ll X^{1+\delta_p}$ prime and $\delta_Y,\delta_p>0$, and assume that $\xi=p/X$ is bounded away from squares of half-integers as $p,X,Y\to\infty$, in the sense that there exists $\delta_\xi>0$ such that $\xi-y^2/4\geq\delta_\xi$ for all $y\in\NN$. The normalized Frobenius trace average over $f\in\msF$ with conductor $N_f\in[X,X+Y]$ is 
\begin{align}
    G(p,X,Y)&=\frac{\sum_{\substack{f\in\msF\\N_f\in[X,X+Y]}}a_f(p)\sqrt{p}}{\sum_{\substack{f\in\msF\\N_f\in[X,X+Y]}}1}\label{eqn:G(p,X,Y)_definition}\\
    &=c(p)\sum_{1\leq y<2\sqrt{\xi}}\delta_y(p)M_y(\xi)+M_-(\xi)+O_{\delta_\xi}\left(X^{-\delta_0+\epsilon}\right),\label{eqn:almost_periodic_murmuration}
\end{align}
where \begin{align}
    \xi&=p/X,\\
    \delta_0&=\min\left\{\frac{13-2\sqrt{37}}{42}-\frac{6}{7}\delta_Y-\frac{12}{7}\delta_p,\delta_Y-\delta_p\right\},\label{eqn:delta_0_definition}\\
    M_y(\xi)&=\frac{11\zeta(2)}{4A}\sqrt{\frac{\xi}{4\xi-y^2}}\,\vartheta(y),\label{eqn:M_y_definition}\\
    M_-(\xi)&=-\frac{11\pi}{12A}\sqrt{\xi},\label{eqn:M_-_definition}\\
    A&=\prod_p\left(1+\frac{p}{(p+1)^2(p-1)}\right),\label{eqn:A_definition}\\
    c(p)&=\frac{p+1}{3p}\prod_{\ell>2,\Legendre{p}{\ell}=1}\left(1-2\ell^{-2}-\frac{2\ell^{-3}}{1-\ell^{-2}}\right),\label{eqn:c(p)_definition}\\
    \delta_y(p)&=\begin{cases}
        \ONE\left\{\Legendre{p}{q}=1\right\} &\text{if $y=q^k$ where $q$ is an odd prime,}\\
        \ONE\{p\equiv3\Mod{4}\} &\text{if $y=2$},\\
        \ONE\{p\equiv5\Mod{8}\} &\text{if $y=4$},\\
        \ONE\{p\equiv1\Mod{8}\} &\text{if $y=2^\nu$ and $\nu\geq3$,}
    \end{cases}\label{eqn:delta_y(p)_definition}\\
    &\text{(extended multiplicatively in $y$)}\nonumber\\
    \vartheta(y)&=2^{\omega(y)+\min(v_2(y),2)}\prod_{q\mid y\,\text{odd}}\left(1+\frac{2q^2+q-1}{q^4-3q^2-2q+2}\right).\label{eqn:vartheta(y)_definition}
\end{align} In particular, for any $\delta_p<\frac{13-2\sqrt{37}}{108}\approx0.0077$, one can find some $\delta_Y>0$ such that $\delta_0>0$.\footnote{Unless otherwise specified, all infinite products are over (subsets of) the primes.}
\end{thm} 

Throughout this paper, we assume the behavior of $p,X,Y$ are as in Theorem \ref{thm:thm1} as they tend to $\infty$.

The main term in \ref{eqn:almost_periodic_murmuration} has dependence on the arithmetic of the prime $p$, and thus is not a murmuration density under the terminologies of \cite{SarnakLetter}. However, this dependence on $p$ is very explicit. We define a function $f:S\to\RR$ where $S\subseteq\NN$ to be \textbf{almost periodic} if for all $\epsilon>0$ there is some modulus $Q$ such that $\abs{f(n)-f(m)}<\epsilon$ whenever $n,m\in S$ and $n-m\in Q\ZZ$. For any given $y$, $c(p)\delta_y(p)$ is an almost periodic function and the main term of \ref{eqn:almost_periodic_murmuration} is a finite sum of smooth functions of $\xi$ multiplied by almost periodic functions in $p$. 

This almost periodic structure of the murmurations for our family $\msF$ is either non-existent or unnoticed in previous families whose murmurations have been computed in \cite{Zubrilina}, \cite{LOP_dirichlet}, \cite{BBLLD_modular_forms}, and \cite{BLLDSHZ_Maass}. One might wonder whether this feature exists for other families of $L$-functions and to what extent. With this formulation one might hope to understand the murmuration phenomenon more precisely without averaging over primes.

To exhibit the connection between murmurations and the 1-level density philosophy of \cite{Katz-Sarnak} in the sense of \ref{eqn:asymps_of_murmuration_function_F_only}, we do need to average over $p$ to extract the murmuration density as defined by \cite{SarnakLetter}. The almost periodic dependence on $p$ is very mild, and thus in principle we expected very little averaging to be required. Precisely, if we assume primes are \textbf{equidistributed in intervals of length $H=o(X)$} in the sense that \begin{equation}\label{eqn:equidistribution_assumption}
    \sum_{\substack{n\in[X,X+H]\\x\equiv a\Mod{q}}}\Lambda(n)=\frac{H}{\varphi(q)}(1+o(1))
\end{equation} for all $a,q$ with $(a,q)=1$, where $\Lambda$ is the von Mangoldt function,\footnote{For example, under the Riemann Hypothesis for Dirichlet $L$-functions, one can take $H=X^{1/2+\epsilon}.$} then averaging over primes in intervals of length $H$ is enough to extract the murmuration density.

\begin{thm}\label{thm:thm2}
Let $P\sim X^{1+\delta_p}$ where $\delta_p\geq0$ be a prime, and choose $\delta_Y$ such that $\delta_0>0$ in Theorem \ref{thm:thm1}. Assume that $\Xi=P/X$ is bounded away from squares of half-integers by at least $\delta_\Xi$ as $P,X,Y\to\infty$, in the sense as in Theorem \ref{thm:thm1}. Let $H=o(P)$ be some parameter such that primes are equidistributed in intervals of length $H$. Then the average of $G(p,X,Y)$ over primes $p\in[P,P+H]$ is \begin{equation}
    G_{\avg}(P,X,Y):=\frac{\sum_{p\in[P,P+H]}G(p,X,Y)}{\sum_{p\in[P,P+H]}1}=\cbarr\sum_{1\leq y<2\sqrt{\Xi}}\Mbarr_y(\Xi)+M_-(\Xi)+o_{\delta_\Xi}(1),
\end{equation} where \begin{align}
    \Xi&=P/N,\\
    \cbarr&=\frac{1}{3}\prod_{\ell>2}\left(1-\ell^{-2}-\frac{\ell^{-3}}{1-\ell^{-2}}\right),\label{eqn:cbarr_definition}\\
    \Mbarr_y(\Xi)&=\frac{11\zeta(2)}{4A}\sqrt{\frac{\Xi}{4\Xi-y^2}}\,\kappa(y),\label{eqn:Mbarr_y_definition}\\
    \kappa(y)&=2^{\ONE\{2\mid y\}}\prod_{q\mid y\text{ odd}}\left(1+\frac{q^2}{q^4-2q^2-q+1}\right).\label{eqn:kappa(y)_definition}
\end{align} and $M_-$ is given in \ref{eqn:M_-_definition} as in Theorem \ref{thm:thm1}. 
\end{thm}

The main term \begin{equation}\label{eqn:murmuration_density}
    M(\Xi)=\cbarr\sum_{1\leq y<2\sqrt{\Xi}}\Mbarr_y(\Xi)+M_-(\Xi)
\end{equation} in Theorem \ref{thm:thm2} is the murmuration density for the family $\msF$. Given any weight function $\Phi$, we can use \ref{eqn:murmuration_function_and_density_relation} to integrate $M(\Xi)$ and obtain the murmuration function $M_\Phi(\Xi)$: \begin{equation}\label{eqn:integrating_murmuration_density}
    M_\Phi(\Xi)=\frac{\int_0^\infty M(\Xi/u)\Phi(u)u^{3/2}\frac{du}{u}}{\int_0^\infty\Phi(u)u^{3/2}\frac{du}{u}}.
\end{equation} 

\begin{thm}\label{thm:thm3}
Let $\Phi:(0,\infty)\to[0,\infty]$ be any compactly supported smooth weight function. Then the murmuration function $M_\Phi(\Xi)$ obtained by integrating the murmuration density \ref{eqn:murmuration_density} via \ref{eqn:integrating_murmuration_density} satisfies $$M_\Phi(0)=0,\AND\lim_{\Xi\to\infty}M_\Phi(\Xi)=-\frac{1}{2}.$$
\end{thm}
We prove Theorem \ref{thm:thm3} in Section \ref{section:asymptotics}.

Numerically, we compute the murmuration function $M_\Phi(\Xi)$ for $\Phi=\chi_{[1,2]}$,\footnote{While the characteristic function $\Phi$ here is not smooth and technically not an appropriate weight function, it is much easier to compute numerics with.} and our analytical computations fit the numerics nicely (see Figure \ref{fig:murmuration_functions_versus_pred}). Moreover, the murmuration function numerically converges to $-\frac{1}{2}$ (see Figure \ref{fig:murmuration_functions_asymp_numerics}), as expected.

A more complete presentation of numerical data is included in Section \ref{B:numerics}.

\begin{figure}[H]
    \centering
    \includegraphics[width=\textwidth]{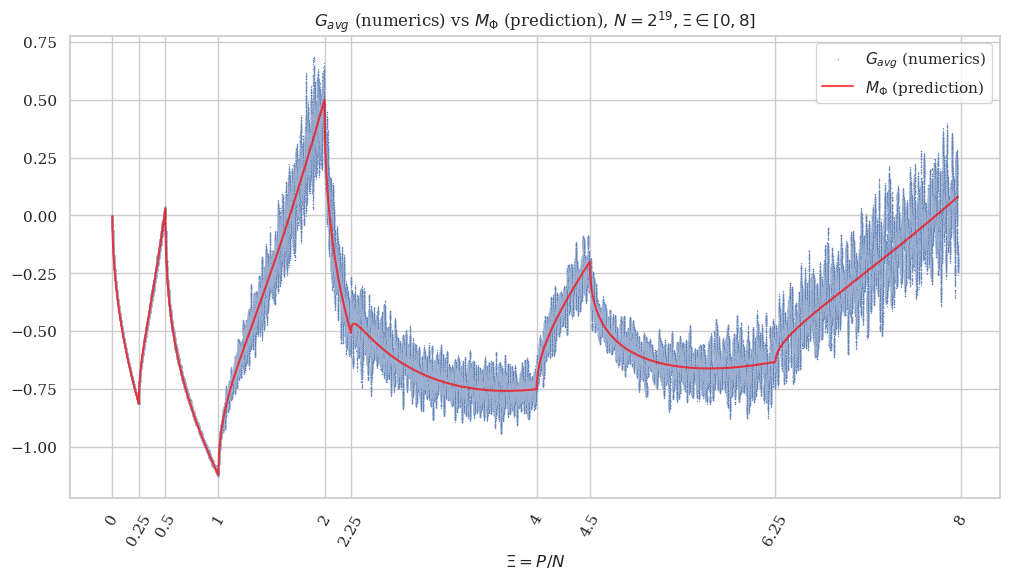}
    \caption{$G_{\text{avg}}(P,N)$ (from numerics, with rolling average over primes in intervals of size $H=P^{0.55}$) vs. murmuration function $M_\Phi(P/N)$ for $\Phi=\chi_{[1,2]}$ (integrated from the density $M(\Xi)$ in \ref{eqn:murmuration_density}), shown for $N=2^{19},\Xi\in[0,8]$.}
    \label{fig:murmuration_functions_versus_pred}
\end{figure}

\begin{figure}[H]
    \centering
    \includegraphics[width=\textwidth]{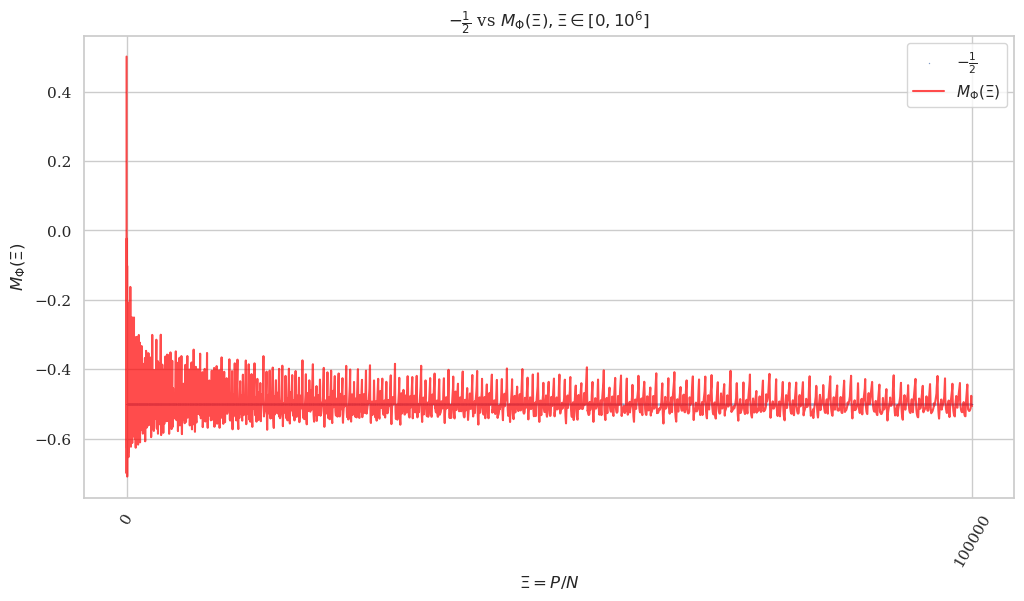}
    \caption{Asymptotic behavior of the murmuration function $M_\Phi(\Xi)$ for $\Phi=\chi_{[1,2]}$.}
    \label{fig:murmuration_functions_asymp_numerics}
\end{figure}

\newtheorem*{ack}{Acknowledgements}
\begin{ack}
I am very grateful to my undergraduate thesis adviser Peter Sarnak for proposing this topic and for his guidance throughout. I also thank Michael Cheng, Alexey Pozdnyakov, Maksym Radziwi\l\l, Kannan Soundararajan, and Liyang Yang for helpful conversations and remarks. I also benefitted greatly from a recent workshop in murmurations, and I am thankful to the organizers Yang-Hui He, Abhiram Kidambi, Kyu-Hwan Lee, and Thomas Oliver, as well as the Simon's Center for Geometry and Physics at Stony Brook University.
\end{ack}
\section{Easy Evaluations}
In this section we simplify \begin{equation}\label{eqn:G(p,X,Y)_sect_2_rewrite}
    G(p,X,Y)=\frac{\sum_{\substack{f\in\msF\\N_f\in[X,X+Y]}}a_f(p)\sqrt{p}}{\sum_{\substack{f\in\msF\\N_f\in[X,X+Y]}}1}
\end{equation} using orthogonality of characters. Computing $G(p,X,Y)$ then becomes a problem in summing class numbers in short intervals. We also evaluate the easier parts of $G(p,X,Y)$.

\subsection{Orthogonality of Characters}
The denominator of $\ref{eqn:G(p,X,Y)_sect_2_rewrite}$ is easily rewritten as $$G_{\denom}(p,X,Y):=\#\{f\in\msF:N_f\in[X,X+Y]\}=\sum_{D\in\mcD}(h(-D)-1),$$ where \begin{equation}\label{eqn:mcD_definition}
    \mcD=\{D\in[X,X+Y]:D\equiv3\Mod{4},D\ \square\text{-free}\}.
\end{equation} The numerator, on the other hand, can be rewritten using orthogonality of characters. Specifically, \begin{align*}
    G_{\num}(p,X,Y)&:=\sum_{\substack{f\in\msF\\N_f\in[X,X+Y]}}a_f(p)\sqrt{p}=\sqrt{p}\sum_{D\in\mcD}\sum_{\substack{\psi:\Cl(K)\to S^1\\\psi\neq\ONE}}a_f(p),
\end{align*} where $K=\QQ(\sqrt{-D})$ and $f=L(s,\psi)$. By \ref{eqn:Frob_traces}, if $\Legendre{-D}{p}=-1$ then the inner sum vanishes, and otherwise it becomes a sum over the nontrivial characters of $\psi(\mfp)+\psi(\mfpbarr)$ if $p\mcO_K=\mfp\mfpbarr$ splits, or of $\psi(\mfp)$ if $p\mcO_K=\mfp^2$ ramifies. By orthogonality of characters, $$\sum_{\substack{\psi\in(\Cl(K))^*\\\psi\neq\ONE}}\psi(\mfp)=\ONE_\mfp h(-D)-1,$$ where $\ONE_\mfp$ is the indicator function for $\mfp$ being a principal ideal. If $p\mcO_K=\mfp^2$ ramifies and $\mfp$ is principal, then $D=-p$ necessarily, and this contributes $\sqrt{p}(h(-p)-1)\ll p^{1+\epsilon}$ to the main term. Therefore, \begin{align*}
    G_{\num}(p,X,Y)&=2\sqrt{p}\sum_{D\in\mcD_p}(\ONE_{\mfp}h(-D)-1)+O\left(p^{1+\epsilon}\right)
\end{align*} where $p\mcO_K=\mfp\mfpbarr$ and \begin{equation}\label{eqn:mcD_p_definition}
    \mcD_p=\left\{D\in\mcD:\Legendre{-D}{p}=1\right\}.
\end{equation}

Now we define \begin{equation}\label{eqn:nu_Dp_definition}
    \nu(D,p):=\#\{x^2+Dy^2=4p:x,y>0\}.
\end{equation}
\counterwithin{thm}{section}
\setcounter{thm}{0}
\begin{prop}
Assume $\Legendre{-D}{p}=1$. Then $\nu(D,p)=\ONE_\mfp$, i.e. it is 1 when $p\mcO_K=(\alpha)(\overline{\alpha})$ splits into principal ideals, and 0 otherwise. 
\end{prop}
\begin{proof}
If $\nu(D,p)>0$ and $x^2+Dy^2=4p$, then $\alpha=\frac{x+y\sqrt{D}}{2}\in\mcO_K$ satisfies $\alpha\alphabarr=p$, so $p\mcO_K$ splits into principal ideals. By unique factorization of ideals, $\nu(D,p)\leq1$.

On the other hand, assume $p\mcO_K=(\alpha)(\alphabarr)$ splits into principal ideals. Since $D>3$ and $\mcO_K^\times=\{\pm1\}$, by possibly changing signs and order, we may assume that $$\alpha=\frac{x+y\sqrt{D}}{2},\qquad\alphabarr=\frac{x-y\sqrt{D}}{2}$$ with $x,y>0$, and consequently $x^2+Dy^2=4p$. Since $\nu(D,p)\leq1$, it is exactly 1.
\end{proof}

Therefore, we can rewrite $G(p,X,Y)$ as \begin{equation}
    G(p,X,Y)=\frac{2\sqrt{p}\sum_{D\in\mcD_p}\left(\nu(D,p)h(-D)-1\right)+O\left(p^{1+\epsilon}\right)}{\sum_{D\in\mcD}(h(-D)-1)},
\end{equation} with $\mcD,\mcD_p,$ and $\nu(D,p)$ given by \ref{eqn:mcD_definition}, \ref{eqn:mcD_p_definition}, and \ref{eqn:nu_Dp_definition}, respectively. For convenience of notation, we further define \begin{align}
    G_{\num}^+(p,X,Y)&:=2\sqrt{p}\sum_{D\in\mcD_p}\nu(D,p)h(-D),\\
    G_{\num}^-(p,X,Y)&:=-2\sqrt{p}\sum_{D\in\mcD_p}1.
\end{align}

\subsection{Evaluation of $G_{\denom}$ and $G_{\num}^-$}\label{section:denom_and_num_minus}
We use Dirichlet's class number formula to express $h(-D)$ in terms of $L(1,\chi_{-D})$: $$G_{\denom}=\sum_{D\in\mcD}(h(-D)-1)=\sum_{D\in\mcD}\frac{\sqrt{D}}{\pi}L(1,\chi_{-D})+O(Y),$$ where recall $\chi_{-D}=\Legendre{-D}{\cdot}$ is the non-principal real Dirichlet character modulo $D$. Moreover, since $\mcD\subset[X,X+Y]$, we have $\sqrt{D}=\sqrt{X}+O\left(X^{-1/2}Y\right)$, and along with $L(1,\chi_{-D})\ll D^{\epsilon}$ we have \begin{equation*}\label{eqn:G_denom_after_smooth_approx}
    G_{\denom}=\frac{\sqrt{X}}{\pi}\sum_{D\in\mcD}L(1,\chi_{-D})+O\left(X^{-1/2+\epsilon}Y^2+Y\right).
\end{equation*} Next we truncate the $L$-function at some cutoff parameter $T$. By Lemma \ref{lem:density_theorem}, we have \begin{align*}
    G_{\denom}&=\frac{\sqrt{X}}{\pi}\sum_{D\in\mcD}\sum_{n\leq T}\frac{\chi_{-D}(n)}{n}+\Err_1\\
    &=\begin{multlined}[t][0.9\textwidth]
        \frac{\sqrt{X}}{\pi}\sum_{n\leq T}\frac{1}{n}\Legendre{-1}{n}\sum_{D\in[X,X+Y]}\mu^2(D)\frac{\chi_{4,1}(D)-\chi_{4,2}(D)}{2}\Legendre{D}{n}+\Err_1,
    \end{multlined}
\end{align*} where $\chi_{4,1},\chi_{4,2}$ are the principal and non-principal characters modulo 4, respectively, and \begin{equation}\label{eqn:section_2_error}
    \Err_1\ll\left(T^{-1}UX^{\frac{8(1-\sigma)}{3-2\sigma}+1}+X^{1/2}Y\left(T^{\sigma-1}+U^{-1}\right)+X^{-1/2}Y^2+Y\right)(TX)^\epsilon,
\end{equation} with $\frac{1}{2}\leq\sigma\leq1$ and $U\ll X$ arbitrary.

Now we apply Lemma \ref{lem:partial_sum_characters_over_squarefree_numbers} to evaluate the inner sum. We need to pick out the principal characters in $D$: this happens for $\chi_{4,1}(D)\Legendre{D}{n}$ when $n=m^2$, in which case the modulus is $2m$. Otherwise, the character is non-principal and its modulus is $O(T)$. Thus \begin{equation*}
    G_{\denom}=\frac{\sqrt{X}}{\pi}\left[\sum_{m\leq\sqrt{T}}\frac{Y}{2m^2}\frac{\eta(2m)}{\zeta(2)}+O\left(T^{1/5+\epsilon}X^{3/5+\epsilon}\right)\right]+\Err_1.
\end{equation*} Then by Lemma \ref{lem:sum_of_eta(2m)},  $$G_{\denom}=\frac{4A}{11\pi\zeta(2)}Y\sqrt{X}+O\left(T^{1/5+\epsilon}X^{11/10+\epsilon}+T^{-1/2}X^{1/2}Y\right)+\Err_1.$$ Setting\footnote{The second term in the error bound of \ref{eqn:G_denom} comes from the terms $T^{1/5}X^{11/10},T^{-1}UX^{\frac{8(1-\sigma)}{3-2\sigma}+1},$ and $U^{-1}X^{1/2}Y$ (up to $\epsilon$-powers), and $T,U,\sigma$ are selected so that they are equal. Note that $\frac{55+\sqrt{1801}}{84}+\frac{1}{7}\approx1.3028<1.5$, so this error term is power-saving.}  \begin{align*}
    \sigma&=\frac{181-\sqrt{1801}}{172}\approx0.8056,\\
    T&=X^{(-187+5\sqrt{1801})/84}Y^{5/7}\approx X^{0.300}Y^{0.714},\\
    U&=X^{(-13-\sqrt{1801})/84}Y^{6/7}\approx X^{-0.660}Y^{0.857},
\end{align*} and applying $Y=o(X)$ to simplify error bounds, we obtain \begin{equation}\label{eqn:G_denom}
    G_{\denom}=\frac{4A}{11\pi\zeta(2)}Y\sqrt{X}+O\left(\left(X^{-1/2}Y^2+X^{(55+\sqrt{1801})/84}Y^{1/7}\right)(XY)^\epsilon\right),
\end{equation} and consequently \begin{equation}\label{eqn:G_denom_inverse}
    G_{\denom}^{-1}=\frac{11\pi\zeta(2)}{4A}\frac{1}{Y\sqrt{X}}+O\left(\left(X^{-3/2}+X^{(-29+\sqrt{1801})/84}Y^{-13/7}\right)(XY)^\epsilon\right).
\end{equation}

Using the same methods, we find \begin{align}
    G_{\num}^-(p,X,Y)&=-2\sqrt{p}\sum_{D\in\mcD_p}1\nonumber\\
    &=-2\sqrt{p}\sum_{D\in[X,X+Y]}\mu^2(D)\frac{\chi_{4,1}(D)-\chi_{4,2}(D)}{2}\frac{\Legendre{-D}{p}+1}{2}+O\left(\sqrt{p}\sum_{\substack{D\in[X,X+Y],p\mid D}}1\right)\nonumber\\
    &=-\frac{\sqrt{p}}{2}\left(\frac{2Y}{3\zeta(2)}+O\left(p^{1/5+\epsilon}X^{3/5+\epsilon}\right)\right)+O\left(p^{1/2}+p^{-1/2}Y\right)\nonumber\\
    &=-\frac{1}{3\zeta(2)}Y\sqrt{p}+O\left(p^{7/10+\epsilon}X^{3/5+\epsilon}\right).\label{eqn:G_num_minus}
\end{align}

\section{Evaluation of $G_{\num}^+$: $y=1$ Terms}\label{section:contribution_y=1}
Now we evaluate $$G_{\num}^+(p,X,Y)=2\sqrt{p}\sum_{D\in\mcD_p}\nu(D,p)h(-D).$$ Recall that $\nu(D,p)=\#\{x^2+Dy^2=4p:x,y>0\}$ and is at most 1. By \textbf{``$\bm{y=y_0}$ terms''} we mean the terms in $G_{\num}^+$ because of a solution $x^2+Dy^2=4p$ with $y=y_0$; we denote the sum of such terms by $G_{\num,y_0}^+$. The goal of this section is to evaluate $G_{\num,1}^+$, and the next section will generalize the result to $G_{\num,y}^+$ for arbitrary $y$. In the $y=1$ term that we compute now, we can already see the almost periodic feature of the murmuration density due to the presence of $c(p)$ (defined in \ref{eqn:c(p)_definition} and below).

We will show:
\begin{thm}\label{thm:thm_y=1}
Assume $p,X,Y$ tend to $\infty$ under the assumptions of Theorem \ref{thm:thm1}; in particular, we assume $\xi=p/X$ is bounded away from $1/4$, i.e. there exists some $\delta_\xi>0$ such that $\xi-1/4\geq\delta_\xi$. We have \begin{equation}\label{eqn:thm_y=1}
    G_{\num,1}^+(p,X,Y)=\frac{c(p)}{\pi\sqrt{4\xi-1}}Y\sqrt{p}+O_{\delta_\xi}\left[\left(p^{5/7}X^{(7+\sqrt{37})/21}Y^{1/7}+\frac{Y^2}{X^{1/2}}\right)(pXY)^\epsilon\right],
\end{equation} where $$c(p)=\frac{p+1}{3p}\prod_{\ell>2,\Legendre{p}{\ell}=1}\left(1-2\ell^{-2}+\frac{2\ell^{-3}}{1-\ell^{-2}}\right)$$ as in \ref{eqn:c(p)_definition}.
\end{thm} The main term is only present for $\xi>1/4$; otherwise, $\nu(D,p)=0$ for all $D$ and $G^+_{\num,1}=0$ trivially. Also note that $\frac{5}{7}+\frac{7+\sqrt{37}}{21}+\frac{1}{7}\approx1.4801<1.5,$ so the first error term is power-saving.

\subsection{Overview of Computation}
Define \begin{equation}\label{eqn:A(p)_B(p)_definition}
    A(p):=\sqrt{4p-(X+Y)},\qquad B(p):=\sqrt{4p-X}.
\end{equation} We note that by our assumptions that $\xi$ is bounded away from $1/4$ and $Y=o(X)$, $A(p)$ is well defined when the parameters are large enough and $A(p),B(p)\asymp\sqrt{p}$. We quickly note that interval $[A(p),B(p)]$ has length \begin{align}
    \Delta(p)&:=B(p)-A(p)\nonumber\\
    &=\sqrt{4p-X}\left(1-\sqrt{1-\frac{Y}{4p-X}}\right)\nonumber\\
    &=\frac{Y}{2\sqrt{4p-X}}+O_{\delta_\xi}\left(p^{-3/2}Y^2\right)\label{eqn:Delta(p)_bound}\\
    &\asymp p^{-1/2}Y.\label{eqn:Delta(p)_asymp}
\end{align}

Now, if $x^2+D=4p$ and $D\in\mcD_p$, then $x\in[A(p),B(p)]$ and odd; on the other hand, if $x\in[A(p),B(p)]$ is odd, then $D=4p-x^2$ is in $\mcD_p$ if square-free. Therefore we define \begin{equation}\label{eqn:ONE_definition}
    \ONE(x):=\ONE\{2\nmid x\},
\end{equation} and the sum of $y=1$ terms is \begin{align*}
    G_{\num,1}^+(p,X,Y)&=2\sqrt{p}\sum_{A(p)\leq x\leq B(p)}\ONE(x)\mu^2(4p-x^2)h(x^2-4p)\\
    &=2\sqrt{p}\sum_{A(p)\leq x\leq B(p)}\ONE(x)\mu^2(4p-x^2)\frac{\sqrt{4p-x^2}}{\pi}\,L(1,\chi_{x^2-4p}),
\end{align*} by Dirichlet's class number formula. For $x\in[A(p),B(p)]$, we have $\sqrt{4p-x^2}=\sqrt{X}+O\left(X^{-1/2}Y\right)$; this along with \ref{eqn:Delta(p)_asymp} and $L(1,\chi_{-D})\ll D^\epsilon$ gives \begin{equation}\label{eqn:y=1_after_smooth_function_approx}
    G_{\num,1}^+=\frac{2\sqrt{pX}}{\pi}\sum_{A(p)\leq x\leq B(p)}\ONE(x)\mu^2(4p-x^2)L(1,\chi_{x^2-4p})+O\left(X^{-1/2+\epsilon}Y^2\right).
\end{equation} Now we introduce the square-free sieve $\mu^2(n)=\sum_{a^2\mid n}\mu(a)$ and truncate the $L$-function at a cutoff parameter $T$ using Lemma \ref{lem:density_theorem}: \begin{align}
    G_{\num,1}^+&=\begin{multlined}[t]
        \frac{2\sqrt{pX}}{\pi}\sum_{A(p)\leq x\leq B(p)}\ONE(x)\left[\sum_{a^2\mid(4p-x^2)}\mu(a)\right]\sum_{n\leq T}\frac{1}{n}\Legendre{x^2-4p}{n}+\Err_2
    \end{multlined}\nonumber\\
    &=\begin{multlined}[t][0.8\textwidth]
        \frac{2\sqrt{pX}}{\pi}\sum_{n\leq T}\frac{1}{n}\sum_{a\leq\sqrt{X+Y}}\mu(a)\sum_{\substack{A(p)\leq x\leq B(p)\\a^2\mid(4p-x^2)}}\ONE(x)\Legendre{x^2-4p}{n}+\Err_2,
    \end{multlined}\label{eqn:G_num,1_after_squarefree_sieve_and_PV}
\end{align} where \begin{equation}\label{eqn:Err_2}
    \Err_2\ll\left(T^{-1}Up^{1/2}X^{\frac{8(1-\sigma)}{3-2\sigma}+1}+X^{1/2}Y\left(T^{\sigma-1}+U^{-1}\right)+X^{-1/2}Y^2\right)(TUpXY)^\epsilon,
\end{equation} with $\frac{1}{2}\leq\sigma\leq1$ and $U\ll X$.

We claim that the main term in \ref{eqn:G_num,1_after_squarefree_sieve_and_PV} is \begin{equation}\label{eqn:main_term}
    \frac{2\sqrt{pX}}{\pi}\Delta(p)\sum_{n=1}^\infty\sum_{a=1}^\infty\frac{\mu(a)}{8n^2a^2}C_{8na^2,p,a}+\Err_3,
\end{equation} with error \begin{equation}\label{eqn:Err_3}
    \Err_3\ll\left(T^{1/2}Vp^{1/2}X^{1/2}+\frac{X^{1/2}Y}{T^{1/2}}+\frac{X^{1/2}Y}{V}+\frac{p^{7/6}X^{1/2}}{V^{2/3}}+\frac{p^{3/2}X^{1/2}}{V^2}+p^{5/6}X^{1/2}\right)(TVpXY)^\epsilon,
\end{equation} where $C_{8na^2,p,a}$ will soon be defined in \ref{eqn:C_definition} and $V$ is another cutoff parameter. Our plan is as follows:
\begin{itemize}
    \item Small $a$, small $n$: in Section \ref{section:small_a_small_n}, we break $A(p)\leq x\leq B(p)$ into intervals of length $8na^2$ to evaluate the inner most sum in \ref{eqn:G_num,1_after_squarefree_sieve_and_PV}. Fixing $n,a$, we expect some equidistribution of $\Legendre{x^2-4p}{n}$ with respect to $x$ satisfying $\ONE(x)=1$ and $a^2\mid(4p-x^2)$; thus we define \begin{equation}\label{eqn:C_definition}
        C_{8na^2,p,a}:=\sum_{\substack{0\leq x<8na^2\\a^2\mid(4p-x^2)}}\ONE(x)\Legendre{x^2-4p}{n},
    \end{equation} and we will show that \begin{multline}\label{eqn:small_a_small_n_goal}
        \frac{2\sqrt{pX}}{\pi}\sum_{n\leq T}\frac{1}{n}\sum_{a\leq V}\mu(a)\sum_{\substack{A(p)\leq x\leq B(p)\\a^2\mid(4p-x^2)}}\ONE(x)\Legendre{x^2-4p}{n}\\=\frac{2\sqrt{pX}}{\pi}\Delta(p)\sum_{n\leq T}\sum_{a\leq V}\frac{\mu(a)}{8n^2a^2}C_{8na^2,p,a}+O\left(T^{1/2+\epsilon}V^{1+\epsilon}\sqrt{pX}\right).
    \end{multline} This case gives the main term.
    \item Small $a$, large $n$: we note that exact evaluation in Section \ref{section:eval_of_C} will show that \begin{equation}\label{eqn:bound_on_C}
        C_{8na^2,p,a}\ll\frac{a^\epsilon n}{k(n)},
    \end{equation} where $k(n)$ is the square-free part of $n$. Using this and \ref{eqn:Delta(p)_asymp}, we have \begin{align}
        \frac{2\sqrt{pX}}{\pi}\Delta(p)\sum_{n>T}\sum_{a\leq V}\frac{\mu(a)}{8n^2a^2}C_{8na^2,p,a}&\ll X^{1/2}Y\sum_{n>T}\sum_{a\leq V}\frac{1}{n^2a^2}\frac{a^\epsilon n}{k(n)}\nonumber\\
        &\ll X^{1/2}Y\sum_{k=1}^\infty\sum_{r>\sqrt{T/k}}\frac{1}{k^2r^2}\nonumber\\
        &\ll T^{-1/2}X^{1/2}Y\label{eqn:small_a_large_n_goal},
    \end{align} where we wrote $n=kr^2$ with $k=k(n)$.
    
    \item Large $a$: similarly, we have \begin{align}
        \frac{2\sqrt{pX}}{\pi}\Delta(p)\sum_{n=1}^\infty\sum_{a>V}\frac{\mu(a)}{8n^2a^2}C_{8na^2,p,a}&\ll X^{1/2}Y\sum_{n=1}^\infty\sum_{a>V}\frac{1}{n^2a^2}\frac{a^\epsilon n}{k(n)}\nonumber\\
        &\ll V^{-1+\epsilon}X^{1/2}Y\sum_{k=1}^\infty\sum_{r=1}^\infty\frac{1}{k^2r^2}\nonumber\\
        &\ll V^{-1+\epsilon}X^{1/2}Y\label{eqn:large_a_goal1}.
    \end{align} On the other hand, the sieve methods in \cite{Friedlander-Iwaniec} (see p.388, Equation 4.1) gives that \begin{align}
        &\quad\ \frac{2\sqrt{pX}}{\pi}\sum_{n\leq T}\frac{1}{n}\sum_{V<a\leq\sqrt{X+Y}}\mu(a)\sum_{\substack{A(p)\leq x\leq B(p)\\a^2\mid(4p-x^2)}}\ONE(x)\Legendre{x^2-4p}{n}\nonumber\\
        &\ll T^\epsilon\sqrt{pX}\sum_{a>V}\#\{A(p)\leq x\leq B(p):x^2-4p\equiv0\Mod{a^2}\}\nonumber\\
        &\ll\sqrt{pX}\left(\frac{Mp^{1/6}}{V^{2/3}}+\frac{p^{1/3}}{V^{1/3}}+\frac{p^{1/2}}{M}+\frac{M^2}{V^2}+\frac{M^{4/3}}{p^{1/3}}\right)(Tp)^\epsilon &&(\text{with }M=B(p)\asymp p^{1/2})\nonumber\\
        &\ll\left(\frac{p^{7/6}X^{1/2}}{V^{2/3}}+\frac{p^{3/2}X^{1/2}}{V^2}+p^{5/6}X^{1/2}\right)(Tp)^\epsilon.\label{eqn:large_a_goal2}
    \end{align} 
\end{itemize} Putting the above together gives \ref{eqn:main_term} and \ref{eqn:Err_3}.

Along with $\Err_2$ and using the asymptotic formula $\Delta(p)=\frac{Y}{2\sqrt{4p-X}}+O_{\delta_\xi}\left(p^{-3/2}Y^2\right)$ from \ref{eqn:Delta(p)_asymp}, we arrive at \begin{equation}\label{eqn:G_num_1_only_before_double_sum}
    G_{\num,1}^+=\frac{Y\sqrt{p}}{\pi\sqrt{4\xi-1}}\sum_{n=1}^\infty\sum_{a=1}^\infty\frac{\mu(a)}{8n^2a^2}C_{8na^2,p,a}+\Err,
\end{equation} with \begin{align}
    \Err&\ll_{\delta_\xi}\Err_2+\Err_3+\sqrt{pX}\cdot p^{-3/2}Y^2\label{eqn:Err_cumulative_Theorem_3.1}\\
    &\ll\begin{multlined}[t]
        \Bigg(T^{1/2}Vp^{1/2}X^{1/2}+\frac{X^{1/2}Y}{T^{1/2}}+\frac{X^{1/2}Y}{V}+\frac{p^{7/6}X^{1/2}}{V^{2/3}}+\frac{p^{3/2}X^{1/2}}{V^2}+p^{5/6}X^{1/2}\\\frac{Up^{1/2}X^{\frac{8(1-\sigma)}{3-2\sigma}+1}}{T}+X^{1/2}Y\left(T^{\sigma-1}+U^{-1}\right)+\frac{Y^2}{X^{1/2}}+\frac{X^{1/2}Y^2}{p}\Bigg)(TUVpXY)^\epsilon
    \end{multlined}\nonumber\\
    &\ll\left(p^{5/7}X^{(7+\sqrt{37})/21}Y^{1/7}+\frac{Y^2}{X^{1/2}}\right)(pXY)^\epsilon\label{eqn:thm1_error}
\end{align} upon choosing\footnote{The main term in $\Err$ comes from the terms $T^{1/2}Vp^{1/2}X^{1/2},\frac{p^{7/6}X^{1/2}}{V^{2/3}},\frac{Up^{1/2}X^{\frac{8(1-\sigma)}{3-2\sigma}+1}}{T},$ and $\frac{X^{1/2}Y}{U}$ (up to $\epsilon$-powers), and $T,U,V,\sigma$ are selected so that they are equal.}  
\begin{align*}
    \sigma&=\frac{46-\sqrt{37}}{42}\approx0.9504,\\
    T&=p^{-13/14}X^{(-35+10\sqrt{37})/42}Y^{5/7}\approx p^{-0.929}X^{0.615}Y^{0.714},\\
    U&=p^{-5/7}X^{(7-2\sqrt{37})/42}Y^{6/7}\approx p^{-0.714}X^{-0.123}Y^{0.857},\\
    V&=p^{19/28}X^{(7-2\sqrt{37})/28}Y^{-3/14}\approx p^{0.679}X^{-0.184}Y^{-0.214}.
\end{align*} This is the error bound in Theorem \ref{thm:thm_y=1}. Finally, we compute in Section \ref{section:eval_of_C} that $$\sum_{n=1}^\infty\sum_{a=1}^\infty\frac{\mu(a)}{8n^2a^2}C_{8na^2,p,a}=c(p),$$ so \ref{eqn:G_num_1_only_before_double_sum} gives the correct main term in Theorem \ref{thm:thm_y=1}.


\subsection{Small $a$, Small $n$}\label{section:small_a_small_n}
The goal of this section is to prove \ref{eqn:small_a_small_n_goal}. We first consider the sum \begin{equation*}
    \sum_{\substack{A(p)\leq x\leq B(p)\\a^2\mid(4p-x^2)}}\ONE(x)\Legendre{x^2-4p}{n}.
\end{equation*} Note that $\ONE(x),\ONE\{a^2\mid(4p-x^2)\}$ and $\Legendre{x^2-4p}{n}$ only depend on $x$ modulo $8na^2$. This motivates the definition \begin{equation}\label{eqn:C_definition_2}
    C_{8na^2,p,a}:=\sum_{\substack{0\leq x<8na^2\\a^2\mid(4p-x^2)}}\ONE(x)\Legendre{x^2-4p}{n}=\begin{cases}
    C_{8n,p}R_{a,p} &\text{if $(n,a)=1$,}\\
    0 &\otherwise,
\end{cases} 
\end{equation} as in \ref{eqn:C_definition}, where \begin{equation}\label{eqn:C_8n,p_definition}
    C_{8n,p}:=\sum_{0\leq x<8n}\ONE(x)\Legendre{x^2-4p}{n}\AND R_{a,p}\equiv\#\{x\mod{a^2}:x^2\equiv4p\Mod{a^2}\}.
\end{equation}

\begin{lem}\label{lem:eval_of_Rap}
Assume that $a$ is odd. We have $$R_{a,p}=\begin{cases}
    2^k &a=q_1^{e_1}\cdots q_k^{e_k},\Legendre{p}{q_j}=1\ \forall j,\\
    0 &\otherwise.
\end{cases}$$ This implies $R_{a,p}\ll a^\epsilon$ uniformly in $p$.
\end{lem}
\begin{proof}
By the Chinese remainder theorem, we have $$R_{a,p}=R_{q_1^{e_1},p}\cdots R_{q_k^{e_k},p}.$$ Note $R_{p^e,p}=0$, since $x^2\equiv 4p\Mod{p^2}$ implies $v_p(x^2)=1$, which is absurd. Otherwise, since $q_j$ is odd, Hensel's Lemma implies $R_{q_j^{e_j},p}=R_{q_j,p}$ which is $2$ if $\Legendre{p}{q_j}=1$ and $0$ otherwise. It follows that $R_{a,p}\leq 2^{\omega(a)}\ll a^\epsilon$.
\end{proof}

We have the following equidistribution result.
\begin{prop}\label{prop:sum_of_Legendre}
We have $$\sum_{\substack{A\leq x\leq B\\a^2\mid(4p-x^2)}}\ONE(x)\Legendre{x^2-4p}{n}=\frac{B-A}{8na^2}C_{8na^2,p,a}+O\left(\frac{a^{\epsilon}n\log n}{\sqrt{k(n)}}\right),$$ where $k(n)$ is the square-free part of $n$.
\end{prop}
\begin{proof}
We may assume $(a,n)=1$ since otherwise all terms vanish. We first show $$\sum_{\substack{A\leq x\leq B\\a^2\mid(4p-x^2)}}\ONE(x)\Legendre{x^2-4p}{n}\ll\frac{a^\epsilon n\log n}{\sqrt{k(n)}}$$ when $B-A\leq 8na^2$. This is achieved using the same techniques as in Lemma 2.1 of \cite{Sarnak}, p.336, with the error bound multiplied by $R_{a,p}$ and the $\log n$ term replaced by $\log(na^2)$ during the application of Lemma 2.2 of \cite{Sarnak}, pp.338-339. Note that $a^2\mid(4p-x^2)$ and $4p-x^2$ is odd when $\ONE(x)=1$, so $a$ is odd and Lemma \ref{lem:eval_of_Rap} applies. The general result follows from splitting $[A,B]$ into intervals of length $8na^2$ each, with at most one exceptional interval.
\end{proof}

Now \ref{eqn:small_a_small_n_goal} follows easily; the main term is obvious, and the error term satisfies \begin{align*}
    \Err&\ll\sqrt{pX}\sum_{n\leq T}\frac{1}{n}\sum_{a\leq V}\frac{a^\epsilon n\log n}{\sqrt{k(n)}}\\
    &\ll V^{1+\epsilon}\sqrt{pX}\sum_{r\leq\sqrt{T}}\sum_{k\leq T/r^2}\frac{(kr)^\epsilon}{\sqrt{k}}\\
    &\ll T^{1/2+\epsilon}V^{1+\epsilon}\sqrt{pX}.
\end{align*}

\subsection{Evaluation and Summation of $C_{8na^2,p,a}$}\label{section:eval_of_C}
The goal of this section is to show $$\sum_{n=1}^\infty\sum_{a=1}^\infty\frac{\mu(a)}{8n^2a^2}C_{8na^2,p,a}=\frac{p+1}{3p}\prod_{\ell>2,\Legendre{p}{\ell}=1}\left(1-2\ell^{-2}-\frac{2\ell^{-3}}{1-\ell^{-2}}\right)=c(p),$$ as in \ref{eqn:c(p)_definition}.

Recall from \ref{eqn:C_definition_2}, \ref{eqn:C_8n,p_definition}, and Lemma \ref{lem:eval_of_Rap} that $$C_{8na^2,p,a}=\begin{cases}
    C_{8n,p}R_{a,p} &\text{if $(n,a)=1$,}\\
    0 &\otherwise,
\end{cases}$$ where $$C_{8n,p}=\sum_{0\leq x<8n}\ONE(x)\Legendre{x^2-4p}{n},\qquad R_{a,p}=\begin{cases}
    2^k &a=q_1^{e_1}\cdots q_k^{e_k},\Legendre{p}{q_j}=1\ \forall j,\\
    0 &\otherwise.
\end{cases}$$ Let $n=2^ep^{e'}p_1^{e_1}\cdots p_k^{e_k}$ (with $e$ and $e'$ possibly zero), then $C_{8n,p}$ splits as a product: \begin{multline}\label{eqn:C_8n_p_product}
    C_{8n,p}={\underbrace{\vphantom{\sum_{x\mod{p^{e_0}}}\Legendre{x^2}{p}^{e_0}}\sum_{x\text{ mod }2^{e+3}\text{ odd}}\Legendre{x^2-4p}{2}^e}_{\text{4 if $e=0$, $4(-2)^e$ otherwise}}}
    \times{\underbrace{\sum_{x\mod p^{e'}}\Legendre{x^2-4p}{p}^{e'}}_{\text{1 if $e'=0$, $(p-1)p^{e'-1}$ otherwise}}}\\
    \times{\prod_{e_j\text{ even}}p_j^{e_j-1}\underbrace{\left[\sum_{x\mod{p_j}}\Legendre{x^2-4p}{p_j}^2\right]}_{p_j-1-\Legendre{p}{p_j}}}
    \times{\prod_{e_j\text{ odd}}p_j^{e_j-1}\underbrace{\left[\sum_{x\mod{p_j}}\Legendre{x^2-4p}{p_j}\right]}_{-1\ (\text{see Lemma \ref{lem:sum_of_(x^2-a/p)_over_residue_class}})}}.
\end{multline} We note that this along with Lemma \ref{lem:eval_of_Rap} gives \ref{eqn:bound_on_C}, i.e. $$C_{8na^2,n,p}\ll\frac{a^\epsilon n}{k(n)},$$ as promised. 

Now we compute \begin{align}
    \sum_{n=1}^\infty\sum_{a=1}^\infty\frac{\mu(a)}{8n^2a^2}C_{8na^2,p,a}&=\sum_{n=1}^\infty\frac{C_{8n,p}}{8n^2}\sum_{a\text{ odd},(n,a)=1}\frac{\mu(a)R_{a,p}}{a^2}\nonumber\\
    &=\sum_{n=1}^\infty\frac{C_{8n,p}}{8n^2}\prod_{\ell\nmid 2n,\Legendre{p}{\ell}=1}\left(1-2\ell^{-2}\right)\nonumber\\
    &=\prod_{\ell>2,\Legendre{p}{\ell}=1}\left(1-2\ell^{-2}\right)\cdot\sum_{n=1}^\infty\frac{C_{8n,p}}{8n^2}\prod_{\ell>2,\ell\mid n,\Legendre{p}{\ell}=1}\frac{1}{1-2\ell^{-2}}.\label{eqn:computation_c(p)_mid}
\end{align} The sum over $n$ is multiplicative, and the local factors are:
\begin{itemize}
    \item At 2 (along with the factor $1/8$): $$\frac{1}{8}\left(4+\sum_{e=1}^\infty\frac{4(-2)^e}{2^{2e}}\right)=\frac{1}{3}.$$ 
    \item At $p$: $$1+\sum_{e'=1}^\infty\frac{(p-1)p^{e'-1}}{p^{2e'}}=\frac{p+1}{p}.$$ 
    \item At $p_j>2,\Legendre{p}{p_j}=1$: $$1+\frac{1}{1-2p_j^{-2}}\left(\sum_{e_j\text{ odd}}\frac{-p_j^{e_j-1}}{p_j^{2e_j}}+\sum_{e_j\text{ even}}\frac{p_j^{e_j-1}(p_j-2)}{p_j^{2e_j}}\right)=1-\frac{2p_j^{-3}}{\left(1-2p_j^{-2}\right)\left(1-p_j^{-2}\right)}.$$
    \item At $p_j,\Legendre{p}{p_j}=-1$: $$1+\frac{1}{1-2p_j^{-2}}\left(\sum_{e_j\text{ odd}}\frac{-p_j^{e_j-1}}{p_j^{2e_j}}+\sum_{e_j\text{ even}}\frac{p_j^{e_j}}{p_j^{2e_j}}\right)=1.$$
\end{itemize} Putting these into \ref{eqn:computation_c(p)_mid} gives $$\sum_{n=1}^\infty\sum_{a=1}^\infty\frac{\mu(a)}{8n^2a^2}C_{8na^2,p,a}=\prod_{\ell>2,\Legendre{p}{\ell}=1}\left(1-2\ell^{-2}\right)\cdot\frac{p+1}{3p}\prod_{\ell>2,\Legendre{p}{\ell}=1}\left(1-\frac{2\ell^{-3}}{\left(1-2\ell^{-2}\right)\left(1-\ell^{-2}\right)}\right)=c(p),$$ as desired.
\section{Evaluation of $G_{\num,y}^+$}
In this section, we generalize the work for $y=1$ and evaluate $G_{\num,y}^+$ for arbitrary $y$. Qualitatively there is only one major difference: note that the norm equation $x^2+Dy^2=4p$ imposes arithmetic restrictions on $p$ modulo primes dividing $y$, e.g. $\Legendre{p}{q}=1$ for all odd prime $q\mid y$, and to account for this we introduce an indicator function $\delta_y(p)$ which is periodic with period $4\rad(y)$. Computationally, the strategy is the same, but the evaluation of the more term is more arduous.

Fixing some $y$, we will show:
\begin{thm}\label{thm:thm_y=y}
Assume $p,X,Y$ tend to $\infty$ under the assumptions of Theorem \ref{thm:thm1}; in particular, we assume $\xi=p/X$ is bounded away from $y^2/4$, i.e. there exists $\delta_\xi>0$ such that $\xi-y^2/4\geq\delta_\xi$, then \begin{equation}\label{eqn:thm_y=y}
    G^+_{\num,y}(p,X,Y)=\frac{\vartheta(y)\delta_y(p)c(p)}{\pi\sqrt{4\xi-y^2}}Y\sqrt{p}+O_{\delta_\xi}\left[y^2\left(p^{5/7}X^{(7+\sqrt{37})/21}Y^{1/7}+\frac{Y^2}{X^{1/2}}\right)(ypXY)^\epsilon\right],
\end{equation} where \begin{align}
    \delta_y(p)&=\begin{cases}
        \ONE\left\{\Legendre{p}{q}=1\right\} &\text{if $y=q^k$ where $q$ is an odd prime,}\\
        \ONE\{p\equiv3\Mod{4}\} &\text{if $y=2$},\\
        \ONE\{p\equiv5\Mod{8}\} &\text{if $y=4$},\\
        \ONE\{p\equiv1\Mod{8}\} &\text{if $y=2^\nu$ and $\nu\geq3$,}
    \end{cases}\\
    &\text{(extended multiplicatively in $y$)}\nonumber\\
    \vartheta(y)&=2^{\omega(y)+\min(v_2(y),2)}\prod_{q\mid y\text{ odd}}\left(1+\frac{2q^2+q-1}{q^4-3q^2-2q+2}\right),
\end{align} as in \ref{eqn:delta_y(p)_definition} and \ref{eqn:vartheta(y)_definition}.
\end{thm}
Again the main term is only present for $\xi>y^2/4$; otherwise, since solutions $x^2+Dy^2=4p$ must satisfy $y^2<\frac{4p}{D}\leq4\xi$, we have $G^+_{\num,y}=0$ trivially. 

\begin{proof}[Proof of Theorem \ref{thm:thm1}]
By the above, $G_{\num,y}^+=0$ if $y\geq2\sqrt{\xi}$. Therefore it follows Theorem \ref{thm:thm_y=y} and results from Section \ref{section:denom_and_num_minus} that \begin{align*}
    G(p,X,Y)&=\frac{\sum_{1\leq y<2\sqrt{\xi}}G_{\num,y}^++G_{\num}^-+O\left(p^{1+\epsilon}\right)}{G_{\denom}}\\
    &=c(p)\sum_{1\leq y<2\sqrt{\xi}}\frac{11\zeta(2)}{4A}\sqrt{\frac{\xi}{4\xi-y^2}}\,\vartheta(y)\delta_y(p)-\frac{11\pi}{12A}\sqrt{\xi}+\Err,
\end{align*} where \begin{align*}
    \Err&\ll_{\delta_\xi}\begin{multlined}[t]
        \frac{1}{Y\sqrt{X}}\left[\frac{p}{X}\left(p^{5/7}X^{(7+\sqrt{37})/21}Y^{1/7}+\frac{Y^2}{X^{1/2}}\right)(pXY)^\epsilon+p^{7/10+\epsilon}X^{3/5+\epsilon}+p^{1+\epsilon}\right]\\+Y\sqrt{p}\left[\left(X^{-3/2}+X^{(-29+\sqrt{1801})/84}Y^{-13/7}\right)(XY)^\epsilon\right]
    \end{multlined}\\
    &\ll\left(\frac{p^{12/7}}{X^{(49-2\sqrt{37})/42}Y^{6/7}}+\frac{pY}{X^2}\right)(pXY)^\epsilon.
\end{align*} 

Now recall $Y\sim X^{1-\delta_Y},p\ll X^{1+\delta_p}$. Plugging into the error bound above, we find that $\Err\ll X^{-\delta_0+\epsilon}$, with $$\delta_0=\min\left\{\frac{13-2\sqrt{37}}{42}-\frac{6}{7}\delta_Y-\frac{12}{7}\delta_p,\delta_Y-\delta_p\right\},$$ as given in \ref{eqn:delta_0_definition}.
\end{proof}

\subsection{Overview of Computation}
We first generalize the notations used in Section \ref{section:contribution_y=1}.

The archimedean restriction on $x$ comes from $D\in[X,X+Y]$, which for $y=1$ was given by \ref{eqn:A(p)_B(p)_definition}; likewise for general $y$ we define \begin{equation}\label{eqn:A_y_B_y_definition}
    A_y(p):=\sqrt{4p-y^2(X+Y)},\qquad B_y(p):=\sqrt{4p-y^2X}.
\end{equation} Again by our assumption that $\xi$ is bounded away from $y^2/4$ and $Y=o(X)$, $A_y(p)$ is well-defined when the parameters are large enough and $A_y(p),B_y(p)\asymp\sqrt{p}$; we also compute \begin{align}
    \Delta_y(p)&:=B_y(p)-A_y(p)\nonumber\\
    &=\sqrt{4p-y^2X}\left(1-\sqrt{1-\frac{y^2Y}{4p-y^2X}}\right)\nonumber\\
    &=\frac{y^2Y}{2\sqrt{4p-y^2X}}+O_{\delta_\xi}\left(y^4p^{-3/2}Y^2\right)\label{eqn:Delta_y(p)_bound}\\
    &\asymp y^2p^{-1/2}Y.\label{eqn:Delta_y(p)_asymp}
\end{align} The arithmetic restriction on $x$ and $p$ comes from $x^2+Dy^2=4p$ and $D\equiv3\Mod{4}$, which for the $y=1$ case was given simply by \ref{eqn:ONE_definition}; for general $y$ we define \begin{equation}\label{eqn:ONE_y_definition}
    \ONE_y(x,p):=\ONE\left\{y^2\mid(4p-x^2),(4p-x^2)/y^2\equiv3\Mod{4}\right\}.
\end{equation} Thus $$\text{$x^2+Dy^2=4p$ and $D\in\mcD_p$}\qquad\Longleftrightarrow\qquad\text{$x\in[A_y(p),B_y(p)],\ \ONE_y(x,p)=1$ and $\mu^2(D)=1.$}$$ We further define \begin{align}
    c_y(p)&:=\sum_{n=1}^\infty\sum_{a=1}^\infty\frac{\mu(a)}{8y^2n^2a^2}C^{(y)}_{8y^2na^2,p,a},\label{eqn:c_y(p)_definition}\\
    C_{8y^2na^2,p,a}^{(y)}&:=\sum_{\substack{0\leq x<8y^2na^2\\a^2\mid(4p-x^2)/y^2}}\ONE_y(x,p)\Legendre{(x^2-4p)/y^2}{n}\label{eqn:C^y_8y^2na^2,p,a_definition}\\
    &=\begin{cases}
        0 &\text{if $(n,a)>1$,}\\
        C^{(y)}_{a_y^2\cdot8y^2n,p}R_{a',p} &\text{if $(n,a)=1$, $a=a_ya'$ with $a_y=(y^\infty,a)$,}
    \end{cases}\label{eqn:C^y_8y^2na^2,p,a_definition_2}\\
    C^{(y)}_{a_y^2\cdot8y^2n,p}&:=\sum_{0\leq x<a_y^2\cdot 8y^2n}\ONE_y(x,p)\ONE\{a_y^2y^2\mid(4p-x^2)\}\Legendre{(x^2-4p)/y^2}{n}.\label{eqn:C^y_a_y^28y^2n,p_definition}
\end{align} 

\begingroup
\allowdisplaybreaks
Now, we proceed as before: \begin{align}
    G_{\num,y}^+&=2\sqrt{p}\sum_{A_y(p)\leq x\leq B_y(p)}\ONE_y(x,p)\mu^2\left((4p-x^2)/y^2\right)h\left(-(4p-x^2)/y^2\right)\nonumber\\
    &=2\sqrt{p}\sum_{A_y(p)\leq x\leq B_y(p)}\ONE_y(x,p)\mu^2\left((4p-x^2)/y^2\right)\frac{\sqrt{(4p-x^2)/y^2}}{\pi}L(1,\chi_{(x^2-4p)/y^2})\nonumber\\
    &=\frac{2\sqrt{pX}}{\pi}\sum_{A_y(p)\leq x\leq B_y(p)}\ONE_y(x,p)\mu^2\left((4p-x^2)/y^2\right)L(1,\chi_{(x^2-4p)/y^2})+\Err_{\ref{eqn:error_1}}\label{eqn:error_1}\\
    &=\begin{multlined}[t]
        \frac{2\sqrt{pX}}{\pi}\sum_{A_y(p)\leq x\leq B_y(p)}\ONE_y(x,p)\left[\sum_{a^2\mid(4p-x^2)/y^2}\mu(a)\right]\left[\sum_{n\leq T}\frac{1}{n}\Legendre{(x^2-4p)/y^2}{n}\right]+\Err_{\ref{eqn:error_2}}
    \end{multlined}\label{eqn:error_2}\\
    &=\frac{2\sqrt{pX}}{\pi}\sum_{n\leq T}\frac{1}{n}\sum_{a\leq\sqrt{X+Y}}\mu(a)\sum_{\substack{A_y(p)\leq x\leq B_y(p)\\a^2\mid(4p-x^2)/y^2}}\ONE_y(x,p)\Legendre{(x^2-4p)/y^2}{n}+\Err_{\ref{eqn:error_3}}\label{eqn:error_3}\\
    &=\frac{2\sqrt{pX}}{\pi}\Delta_y(p)\sum_{n=1}^\infty\sum_{a=1}^\infty\frac{\mu(a)}{8y^2n^2a^2}C_{8y^2na^2,p,a}^{(y)}+\Err_{\ref{eqn:error_4}}\label{eqn:error_4}\\
    &=\frac{y^2c_y(p)}{\pi\sqrt{4\xi-y^2}}Y\sqrt{p}+\Err_{\ref{eqn:error_5}}.\label{eqn:error_5}
\end{align} Here all errors are the same as in the previous section but with some additional dependence on $y$, which will be examined in detail in Section \ref{section:updated_error_bound}.
\endgroup

Then it suffices to show \begin{equation}\label{eqn:c_y(p)_evaluation}
    c_y(p)=\frac{\vartheta(y)}{y^2}\delta_y(p)c(p)
\end{equation} for each $y$. This is similar to the computation of $c(p)$ in Section \ref{section:eval_of_C} but requires more care. We dedicate Sections \ref{section:local_factor_of_C^y} and \ref{section:eval_of_cy(p)} to this computation.

\subsection{New Error Bounds}\label{section:updated_error_bound}
The error in \ref{eqn:error_1} comes from the approximation $\sqrt{(4p-x^2)/y^2}=\sqrt{X}+O\left(X^{-1/2}Y\right)$; since $(4p-x^2)/y^2\in[X,X+Y]$, this error term is identical to the $y=1$ case. During the propagation, this is multiplied by the interval length $\Delta_y(p)\asymp y^2p^{-1/2}Y$ as in \ref{eqn:Delta_y(p)_bound} (cf. $\Delta(p)\asymp p^{-1/2}Y$ as in \ref{eqn:Delta(p)_bound}), so we pick up a factor of $y^2$ compared with the error in \ref{eqn:y=1_after_smooth_function_approx}: $$\Err_{\ref{eqn:error_1}}\ll y^2X^{-1/2+\epsilon}Y^2.$$

The additional error in \ref{eqn:error_2} (and in \ref{eqn:error_3}, since the latter is obtained by only switching order of summation in the main term) comes from the truncation of the $L$-function and Lemma \ref{lem:density_theorem}. The bound in Lemma \ref{lem:density_theorem} has two parts, the first depending on the maximum discriminant in $\msD$, and the second depending on $\abs{\msD}$; thus the first is unchanged and the second is multiplied by a factor of $y^2$. For simplicity, we use $$\Err_{\ref{eqn:error_3}}=\Err_{\ref{eqn:error_2}}\ll y^2\Err_2,$$ where $\Err_2$ was given in \ref{eqn:Err_2}. 

The additional error in \ref{eqn:error_4} is a bit more involved. As in the $y=1$ case, we split the error into three regions. All errors are propagated through $\Delta_y(p)$ which gives a factor of $y^2$.
\begin{itemize}
    \item Small $a$, small $n$: the analog of Proposition \ref{prop:sum_of_Legendre} holds with the $\log n$ in the error bound replaced by $\log(y^2na^2)$, which amounts to an additional $y^\epsilon$ factor compared with the error in \ref{eqn:small_a_small_n_goal}: $$\Err\ll y^{2+\epsilon}T^{1/2+\epsilon}V^{1+\epsilon}\sqrt{pX}.$$
    \item Small $a$, large $n$: exact evaluation in Section \ref{section:local_factor_of_C^y} will show that $$C_{8y^2na^2,p,a}^{(y)}\ll\frac{a^\epsilon n}{k(n)},$$ in analogy to \ref{eqn:bound_on_C}. Thus there is no additional factor compared with the error in \ref{eqn:small_a_large_n_goal}: $$\Err\ll y^2T^{-1/2}X^{1/2}Y.$$
    \item Large $a$: for the first error, there is no additional factor compared with that in \ref{eqn:large_a_goal1}: $$\Err\ll y^2V^{-1+\epsilon}X^{1/2}Y.$$ There is also no additional factor compared with the error in \ref{eqn:large_a_goal2} since $B_y(p)\asymp p^{1/2}$ as well. Thus $$\Err\ll y^2\left(\frac{p^{7/6}X^{1/2}}{V^{2/3}}+\frac{p^{3/2}X^{1/2}}{V^2}+p^{5/6}X^{1/2}\right)(Tp)^\epsilon.$$
\end{itemize} Putting these together, we get $$\Err_{\ref{eqn:error_4}}\ll y^{2+\epsilon}\Err_3+\Err_{\ref{eqn:error_3}},$$ where $\Err_3$ was given in \ref{eqn:Err_3}. 

Finally, there is an additional error term from applying the asymptotic formula of $\Delta_y(p)$ from \ref{eqn:Delta_y(p)_bound} whose error term is $y^4$ times that of \ref{eqn:Delta(p)_asymp}, and this error is propagated through $c_y(p)\ll y^{-2+\epsilon}$ (which follows from exact evaluation in Sections \ref{section:local_factor_of_C^y} and \ref{section:eval_of_cy(p)} and $\vartheta(y)\ll y^\epsilon$), so this particular term is also multiplied by a factor of $y^{2+\epsilon}$. Thus, the total error is: \begin{align*}
    \Err_{\ref{eqn:error_5}}&\ll_{\delta_\xi}y^{2+\epsilon}\Err_{\ref{eqn:thm1_error}}\ll y^2\left(p^{5/7}X^{(7+\sqrt{37})/21}Y^{1/7}+\frac{Y^2}{X^{1/2}}\right)(ypXY)^\epsilon,
\end{align*} and this is the error bound in Theorem \ref{thm:thm_y=y}.

\subsection{Local Product for $C^{(y)}_{a_y^2\cdot8y^2n,p}$}\label{section:local_factor_of_C^y}
From now on, our goal is to prove \ref{eqn:c_y(p)_evaluation}. We follow the strategy in Section \ref{section:eval_of_C}, and we start with the local product for $C^{(y)}_{a_y^2\cdot 8y^2n,p}$.

Recall $a_y=(y^\infty,a)$ and $$C^{(y)}_{a_y^2\cdot8y^2n,p}=\sum_{0\leq x<a_y^2\cdot 8y^2n}\ONE\left\{a_y^2y^2\mid(4p-x^2),(4p-x^2)/y^2\equiv3\Mod{4}\right\}\Legendre{(x^2-4p)/y^2}{n}.$$ The length of the interval of summation in the definition of $C^{(y)}_{a_y^2\cdot 8y^2n,p}$ guarantees that it splits as a product. Precisely, if $$y=2^{\nu}q_1^{\nu_1}\cdots q_r^{\nu_r},\qquad a_y=2^\mu q_1^{\mu_1}\cdots q_r^{\mu_r},\qquad n=2^ep^{e'}q_1^{f_1}\cdots q_r^{f_r}p_1^{e_1}\cdots p_k^{e_k}$$ (since we will require $y<2\sqrt{\xi}=2\sqrt{p/X}$, we may assume that $p\nmid y$), then \begin{equation}\label{eqn:C^y_ay2_8y2n_p_product}
    C^{(y)}_{a_y^2\cdot 8y^2n,p}=\sigma(2)\sigma(p)\prod_{i=1}^r\sigma_{I}(q_i)\cdot\prod_{j=1}^k\sigma_{\RomanII}(p_j),
\end{equation} with \begin{align}
    \sigma(2)&=\sum_{x\mod 2^{e+2\mu+2\nu+3}}\ONE\left\{2^{2\mu+2\nu}\mid(4p-x^2),(4p-x^2)/y^2\equiv3\Mod{4}\right\}\Legendre{(x^2-4p)/2^{2\nu}}{2}^e,\label{eqn:sigma_2_definition}\\
    \sigma(p)&=\sum_{x\mod p^{e'}}\Legendre{x^2-4p}{p}^{e'},\label{eqn:sigma_p_definition}\\
    \sigma_I(q_i)&=\sum_{x\mod f_i+2\mu_i+2\nu_i}\ONE\left\{q_i^{2\mu_i+2\nu_i}\mid(4p-x^2)\right\}\Legendre{(x^2-4p)/q_i^{2\nu_i}}{q_i}^{f_i},\label{eqn:sigma_I_definition}\\
    \sigma_{\RomanII}(p_j)&=\sum_{x\mod p_j^{e_j}}\Legendre{x^2-4p}{p_j}^{e_j},\label{eqn:sigma_II_definition}
\end{align} so $\sigma_I$ is the local factor at the odd places dividing $y$ and $\sigma_{\RomanII}$ at those not dividing $y$. Note that each $$\sigma_{\bullet}(\star)=\sigma_{\bullet}(\star,v_\star(y),v_\star(a_y),v_\star(n))$$ is a function of not only the place $\star$ but also of the valuations of $y,a_y,n$ there; we simplify the notation here but will use the complete version later to avoid confusion.

As in \ref{eqn:C_8n_p_product}, we have
\begin{align*}
    \sigma(p,0,0,e')&=\begin{cases}
        \FixedSize{1} &\text{if $e'=0$,}\\
        \FixedSize{(p-1)p^{e'-1}} &\text{if $e'>0$},
    \end{cases}\\
    \sigma_{\RomanII}(p_j,0,0,e_j)&=\begin{cases}
        p_j^{e_j-1}\left(p_j-1-\Legendre{p}{p_j}\right) &\text{if $e_j$ even,}\\
        -p_j^{e_j-1} &\text{if $e_j$ odd}.
    \end{cases}
    \intertext{It's also not hard to compute that}
    \sigma(2,\nu,\mu,e)&=\begin{cases}
        \FixedSize{4} &\text{if $\nu=0$ and $e=0$,}\\
        \FixedSize{4(-2)^e} &\text{if $\nu=0$ and $e>0$,}\\
        \FixedSize{2^{e+3}} &\text{if $\nu=1$, $e\geq0$ even, and $p\equiv3\Mod{4}$,}\\
        \FixedSize{2^{e+4}} &\text{if $\nu=2$, $e\geq0$ even, and $p\equiv5\Mod{8}$,}\\
        \FixedSize{2^{e+4}} &\text{if $\nu\geq3$, $e\geq0$ even, and $p\equiv1\Mod{8}$,}\\
        0 &\otherwise.
    \end{cases}\\
    \sigma_{I}(q_i,\nu_i,\mu_i,f_i)&=\begin{cases}
        \FixedSize{2} &\text{if $\Legendre{p}{q_i}=1,f_i=0$,}\\
        \FixedSize{2(q_i-1)q_i^{f_i-1}} &\text{if $\Legendre{p}{q_i}=1,f_i>0$ even, and $\mu_i=0$,}\\
        \FixedSize{0} &\otherwise.
    \end{cases}
\end{align*} We note that $C^{(y)}_{a_y^2\cdot 8y^2n,p}=0$ if $\delta_y(p)=0$, i.e. $C^{(y)}_{a_y^2\cdot 8y^2n,p}=0$ unless $p$ is a square modulo every odd prime factor of $y$, and moreover $p$ is in a certain residue class modulo either 4 or 8 if $y$ is even.

\subsection{Evaluation of $c_y(p)$}\label{section:eval_of_cy(p)}
In this section we complete the proof of \ref{eqn:c_y(p)_evaluation}, i.e. $$c_y(p)=\frac{\vartheta(y)}{y^2}\delta_y(p)c(p).$$

First assume $y$ is not a power of 2. Then with the notation from the previous section, there is some odd prime $q_1\mid y$ with $\nu_1=v_{q_1}(y)$ and $f_1=v_{q_1}(n)$; let $\ty=y/q_1^{\nu_1}$ be the non-$q_1$ part of $y$. We compare $c_y(p)$ to $c_{\ty}(p)$. If $\Legendre{p}{q_1}\neq1$ then $\delta_y(p)=c_y(p)=0$. Otherwise, $\delta_y(p)=\delta_{\ty}(p)$; assuming both are nonzero, from \ref{eqn:C^y_ay2_8y2n_p_product} we have \begin{equation}\label{eqn:relation_between_C^y_and_C^ytilde}
    \frac{C^{(y)}_{8y^2n,p}}{\sigma_{I}(q_1,\nu_1,0,f_1)}=\frac{C^{(\ty)}_{8\ty^2n,p}}{\sigma_{\RomanII}(q_1,0,0,f_1)}\AND\frac{C^{(y)}_{a_y^2\cdot8y^2n,p}}{\sigma_I(q_1,\nu_1,0,f_1)}=\frac{C^{(\ty)}_{a_{\ty}^2\cdot8\ty^2n,p}}{\sigma_{\RomanII}(q_1,0,0,f_1)},
\end{equation} corresponding to the cases where $(y,a)=1$ and $(y,a)\neq1$, respectively. 

Now $$c_y(p)=\sum_{n=1}^\infty\sum_{a=1}^\infty\frac{\mu(a)}{8y^2n^2a^2}C_{8y^2na^2,p,a}^{(y)}$$ is defined as a sum over $n$ and $a$, and the summand vanishes unless $(n,a)=1$, $a$ odd and squarefree. We split this sum into $q_1\nmid a$ and $q_1\Vert a$ parts, and then use \ref{eqn:C^y_8y^2na^2,p,a_definition_2}, \ref{eqn:relation_between_C^y_and_C^ytilde}:
\begin{align*}
    c_y(p)&=\sum_{n=1}^\infty\frac{C^{(y)}_{8y^2n,p}}{8y^2n^2}\sum_{a,(2q_1n,a)=1}\frac{\mu(a)R_{a,p}}{a^2}+\sum_{n,q_1\nmid n}\frac{C_{a_y^2\cdot8y^2n,p}^{(y)}}{8y^2n^2}\sum_{a,q_1\Vert a,(2n,a)=1}\frac{\mu(a)R_{a/q_1,p}}{a^2}\\
    &=\begin{multlined}[t]
        \sum_{n=1}^\infty\frac{C^{(\ty)}_{8\ty^2n,p}}{8\ty^2n^2}\frac{\sigma_I(q_1,\nu_1,0,f_1)}{q_1^{2\nu_1}\sigma_{\RomanII}(q_1,0,0,f_1)}\sum_{a,(2q_1n,a)=1}\frac{\mu(a)R_{a,p}}{a^2}\\+\sum_{n,q_1\nmid n}\frac{C_{a_{\ty}^2\cdot8\ty^2n,p}^{(\ty)}}{8\ty^2n^2}\frac{\sigma_I(q_1,\nu_1,0,0)}{q_1^{2\nu_1}\sigma_{\RomanII}(q_1,0,0,0)}\left(-\frac{1}{q_1^2}\right)\sum_{\ta,(2q_1n,\ta)=1}\frac{\mu(\ta)R_{\ta,p}}{\ta^2},
    \end{multlined}\\
    \intertext{and similarly}
    c_{\ty}(p)&=\sum_{n=1}^\infty\frac{C^{\ty}_{8\ty^2n,p}}{8\ty^2n^2}\sum_{\substack{a,(2q_1n,a)=1}}\frac{\mu(a)R_{a,p}}{a^2}+\sum_{n,q_1\nmid n}\frac{C^{(y)}_{a_{\ty}^2\cdot 8\ty^2n,p}}{8\ty^2n^2}\left(-\frac{1}{q_1^2}\right)\sum_{\ta,(2q_1n,\ta)=1}\frac{\mu(\ta)R_{\ta,p}}{\ta^2}.
\end{align*} The sum over $n$ also splits multiplicatively, and we see that the local factors for $c_y(p)$ and $c_{\ty}(p)$ are the same except for at $q_1$. Using \ref{eqn:sigma_I_definition} and \ref{eqn:sigma_II_definition}, we can compute that the factor for $c_y(p)$ at $q_1$ is $$\frac{1}{q_1^{2\nu_1}}\left(\sum_{f_1=0}^\infty\frac{\sigma_I(q_1,\nu_1,0,f_1)}{q_1^{2f_1}}-\frac{2}{q_1^2}\right)=\frac{1}{q_1^{2\nu_1}}\frac{2(q_1^3+q_1^2-1)}{q_1^2(q_1+1)},$$ and that for $c_{\ty}(p)$ is $$1+\sum_{f_1=1}^\infty\frac{\sigma_{\RomanII}(q_1,0,0,f_1)}{q_1^{2f_1}}-\frac{2}{q_1^2}=\frac{q_1^4-3q_1^2-2q_1+2}{q_1^2(q_1^2-1)}.$$ Thus \begin{align*}
    c_y(p)&=c_{\ty}(p)\cdot\frac{2(q_1-1)(q_1^3+q_1^2-1)}{q_1^{2\nu_1}(q_1^4-3q_1^2-2q_1+2)}\\
    &=c_{\ty}(p)\cdot\frac{2}{q_1^{2\nu_1}}\left(1+\frac{2q_1^2+q_1-1}{q_1^4-3q_1^2-2q_1+2}\right)\\
    &=c_{\ty}(p)\cdot\frac{{\vartheta(q_1^{\nu_1})}}{q_1^{2\nu_1}}.
\end{align*} 


\noindent Since $q_1^{\nu_1}$ is the $q_1$-part of $y$ and $\vartheta(y)/y^2$ is multiplicative, we immediately have $c_y(p)=\vartheta(\ty)c_{2^\nu}(p)$ where $\ty$ is the odd part of $y$ and $2^\nu$ is the 2-part of $y$. 

Now it suffices to show $c_y(p)=\frac{\vartheta(y)}{y^2}\delta_y(p)c(p)$ in the case where $y=2^\nu$ is a power of 2. Again $\delta_y(p)=0$ implies $c_y(p)=0$; otherwise, the local factor for $c_y(p)$ at 2 is:
\begin{itemize}
    \item $\nu=0$: we recall from Section \ref{section:eval_of_C} that the the local factor for $c(p)$ at 2 is $1/3$.
    \item $\nu=1$: $$\frac{1}{8\cdot 2^{2\nu}}\sum_{e=0}^\infty\frac{\sigma(2,\nu,\mu,e)}{2^{2e}}=\frac{1}{8\cdot 4}\sum_{e\geq0\text{ even}}\frac{2^{e+3}}{2^{2e}}=\frac{1}{3},$$
    \item $\nu\geq2$: $$\frac{1}{8\cdot 2^{2\nu}}\sum_{e=0}^\infty\frac{\sigma(2,\nu,\mu,e)}{2^{2e}}=\frac{1}{8\cdot 2^{2\nu}}\sum_{e\geq0\text{ even}}\frac{2^{e+4}}{2^{2e}}=\frac{1}{3\cdot 2^{2\nu-3}}.$$
\end{itemize} It follows that (assuming $y=2^\nu$, $\delta_y(p)\neq0$) \begin{align*}
    c_y(p)&=\begin{cases}
        c(p) &\text{if $\nu=0,1$,}\\
        2^{-2\nu+3}c(p) &\text{if $\nu\geq2$},
    \end{cases}\\
    &=\frac{2^{\omega(y)}2^{\min(\nu,2)}}{y^2}c(p)\\
    &=\frac{\vartheta(y)}{y^2}c(p),
\end{align*} as desired. This complete the proof of \ref{eqn:c_y(p)_evaluation} and thus of Theorem \ref{thm:thm_y=y}.

\section{Averages over $p$ and the Murmuration Density}
In this section we conduct a second average over $p$ to obtain Theorem \ref{thm:thm2} and the murmuration density as defined in \cite{SarnakLetter}, which requires the murmuration density to be a smooth function of $\xi=p/N$. This also helps us establish the relation to the 1-level density conjecture in the next section. We introduce the following notation: if $f(p)$ is a function on the primes and $[A,B]$ is an interval, we let \begin{equation}
    \Exp_{p\in[A,B]}f(p):=\frac{\sum_{p\in[A,B]}f(p)}{\sum_{p\in[A,B]}1}
\end{equation} denote the average of $f(p)$ over primes in $[A,B]$.

Recall from section \ref{chapter:introduction} that we say that \textit{primes are equidistributed in intervals of length $H=o(X)$} if $$\sum_{\substack{x\in[X,X+H]\\x\equiv a\Mod{q}}}\Lambda(x)=\frac{H}{\varphi(q)}(1+o(1))$$ for all $a,q$ with $(a,q)=1$, where $$\Lambda(n)=\begin{cases}
    \log p &\text{if $n=p^k$ is a prime power,}\\
    0 &\otherwise,
\end{cases}$$ is the von Mandolgt function. This is equivalent to $$\#\{p\in[X,X+H]:p\equiv a\Mod{q}\}=\frac{\li(X+H)-\li(X)}{\varphi(q)}(1+o(1)),$$ where $$\li(X)=\int_0^X\frac{dt}{\log t}$$ is the logarithmic integral. For example, under standard assumptions such as the Generalized Riemann Hypothesis for Dirichlet $L$-functions, we may take $H=X^{1/2+\epsilon}$ (see \cite{Davenport}, p.125); however we will not require the power of GRH beyond the prime number theorem in arithmetic progressions, and we expect the prime number theorem to hold for much smaller $H$.

\subsection{Averages of $c(p)$ over $p$}
We first average $$c(p)=\frac{p+1}{3p}\prod_{\ell>2,\Legendre{p}{\ell}=1}\left(1-2\ell^{-2}-\frac{2\ell^{-3}}{1-\ell^{-2}}\right)$$ over intervals.
\begin{prop}\label{prop:average_of_c(p)}
Assume primes are equidistributed in intervals of length $H=o(X)$. We have $$\Exp_{p\in[X,X+H]}c(p)=\cbarr+o(1),$$ where \begin{equation}
    \cbarr=\frac{1}{3}\prod_{\ell>2}\left(1-\ell^{-2}-\frac{\ell^{-3}}{1-\ell^{-2}}\right)
\end{equation} as given in \ref{eqn:cbarr_definition}.
\end{prop}
\begin{proof}
Define $$c_M(p):=\frac{1}{3}\prod_{2<\ell<M,\Legendre{p}{\ell}=1}\left(1-2\ell^{-2}-\frac{2\ell^{-3}}{1-\ell^{-2}}\right)$$ where $M$ is some cutoff parameter. Let $S=\{\ell:2<\ell<M\}$ denote the set of odd primes less than $M$, and let $q=\prod_{\ell\in S}\ell$ denote their product. 

The prime number theorem in arithmetic progressions implies \begin{align*}
    \sum_{p\in[X,X+H]}c_M(p)&=\sum_{a\Mod{q}}\#\{p\in[X,X+H]:p\equiv a\Mod{q}\}\cdot\frac{p+1}{3p}\prod_{\ell\in S,\Legendre{a}{\ell}=1}\left(1-2\ell^{-2}-\frac{2\ell^{-3}}{1-\ell^{-2}}\right)\\
    &=\sum_{a\Mod{q}}\frac{\li(X+H)-\li(X)}{\varphi(q)}\cdot\frac{1}{3}\prod_{\ell\in S,\Legendre{a}{\ell}=1}\left(1-2\ell^{-2}-\frac{2\ell^{-3}}{1-\ell^{-2}}\right)(1+o(1)).
\end{align*} Now by the Chinese remainder theorem, given any sequence $\epsilon=(\epsilon_1,\cdots,\epsilon_k)\in\{\pm1\}^k$, there are exactly $\prod_{\ell\in S}\frac{\ell-1}{2}=2^{-k}\varphi(q)$ primitive residue classes $a$ modulo $q$ such that $\Legendre{a}{\ell_i}=\epsilon_i$ for all $i$. Thus \begin{align}
    \sum_{p\in [X,X+H]}c_M(p)&=\frac{\li(X+H)-\li(X)}{\varphi(q)}\cdot2^{-k}\varphi(q)\sum_{(\epsilon_1,\cdots,\epsilon_k)\in\{\pm1\}^k}\frac{1}{3}\prod_{i\leq k,\epsilon_i=1}\left(1-2\ell_i^{-2}-\frac{2\ell_i^{-3}}{1-\ell_i^{-2}}\right)(1+o(1))\nonumber\\
    &=\frac{\li(X+H)-\li(X)}{3}\prod_{\ell\in S}\left(1-\ell^{-2}-\frac{\ell^{-3}}{1-\ell^{-2}}\right)(1+o(1)).\label{eqn:avg_c(p)_num}
\end{align}  For the denominator, the prime number theorem gives $$\sum_{p\in[X,X+H]}1\ll(\li(X+H)-\li(X))(1+o(1)).$$ Thus $$\Exp_{p\in[X,X+H]}c_M(p)=\frac{1}{3}\prod_{\ell\in S}\left(1-\ell^{-2}-\frac{\ell^{-3}}{1-\ell^{-2}}\right)+o(1).$$

Now let $M$ tend to $\infty$ with $X$ (in arbitrary fashion, in fact), so $$\abs{c(p)-c_M(p)}\ll\sum_{\ell\geq M}\frac{1}{\ell^2}=o(1),$$ and similarly $$\abs{\frac{1}{3}\prod_{\ell\in S}\left(1-\ell^{-2}-\frac{\ell^{-3}}{1-\ell^{-2}}\right)-\cbarr}=o(1).$$ Combining the error bounds establishes the Proposition.
\end{proof}

\subsection{Averages of $G(p,X,Y)$ over $p$}
Recall from Theorem \ref{thm:thm1} that \begin{equation*}
    G(p,X,Y)=c(p)\sum_{1\leq y<2\sqrt{\xi}}\delta_y(p)M_y(\xi)+M_-(\xi)+O_{\delta_\xi}\left(X^{-\delta_0+\epsilon}\right).
\end{equation*} The goal of this section is to average over $p\in[P,P+H]$ assuming primes are equidistributed in intervals of length $H=o(X)$, thereby obtaining Theorem \ref{thm:thm2}.

By the assumption that $\Xi=P/X$ is bounded away from squares of half-integers and $H=o(X)$, Theorem \ref{thm:thm1} applies to all primes $p\in[P,P+H]$ when the parameters are large enough. The average of $M_-(\Xi)$ is simply $M_-(\Xi)$ as it is a smooth function and has bounded derivative. Thus it suffices to show that \begin{equation}\label{eqn:averaging_1}
    \Exp_{p\in[P,P+H]}c(p)\delta_y(p)M_y(\xi)=\cbarr\Mbarr_y(\Xi)+o_{\delta_\Xi}(1)
\end{equation} for each $1\leq y<2\sqrt{\Xi}$. 

The function $$\psi(\xi)=\frac{11\zeta(2)}{4A}\sqrt{\frac{\xi}{4\xi-y^2}}$$ is smooth for $\xi\geq\Xi$, and a derivative bound gives $$\psi(\Xi)-\psi(\xi)\ll_{\delta_\Xi}\frac{H}{X}=o(1)$$ for $p\in[P,P+H]$. Consequently, \begin{align}
    \Exp_{p\in[P,P+H]}c(p)\delta_y(p)M_y(\xi)&=\Exp_{p\in[P,P+H]}c(p)\delta_y(p)\vartheta(y)\psi(\xi)\nonumber\\
    &=\vartheta(y)\psi(\Xi)\Exp_{p\in[P,P+H]}c(p)\delta_y(p)+o_{\delta_\Xi}(1).\label{eqn:averaging_2}
\end{align} 

The average of $c(p)\delta_y(p)$ can be handled as that for $c(p)$. The difference is that we include all primes dividing $y$ to the modulus considered (see proof of Proposition \ref{prop:average_of_c(p)}), and the main term is multiplied by a factor of $$\eta(y):=2^{-\omega(y)+\star}\prod_{q\mid y\text{ odd}}\frac{1-2q^{-2}-\frac{2q^{-3}}{1-q^{-2}}}{1-q^{-2}-\frac{q^{-3}}{1-q^{-2}}},\qquad\star=\begin{cases}0 &\text{if }4\nmid y,\\-1 &\text{if }4\mid y.\end{cases}$$ In fact, for every odd prime $q\mid y$, the factor $1-q^{-2}-\frac{q^{-3}}{1-q^{-2}}$ is replaced by $\frac{1}{2}\left(1-2q^{-2}-\frac{2q^{-3}}{1-q^{-2}}\right)$ since we are only summing over the $p$'s that are squares modulo $q$; at the place 2, we get a factor of $1/2$ if $2\Vert y$ since we are only summing over $p\equiv3\Mod{4}$, and a factor of $1/4$ if $4\mid y$ since we are only summing over $p$ in one primitive residue class modulo 8 (either 1 or 5). Thus \begin{equation}\label{eqn:averaging_3}
    \Exp_{p\in[P,P+H]}c(p)\delta_y(p)=\eta(y)\cbarr+o\left(2^{-\omega(y)}\right).
\end{equation} Finally we put \ref{eqn:averaging_3} into \ref{eqn:averaging_2} to obtain \ref{eqn:averaging_1}; the main term comes from noting $\vartheta(y)\eta(y)=\kappa(y)$, and the error comes from noting $\vartheta(y)\ll 2^{\omega(y)}$ and $\psi(\Xi)\ll_{\delta_\Xi}1$. 

This completes the proof of Theorem \ref{thm:thm2}, and we have thus obtained the murmuration density $$M(\Xi)=\cbarr\sum_{1\leq y<2\sqrt{\Xi}}\Mbarr_y(\Xi)+M_-(\Xi)$$ of the family $\msF$, with $\cbarr,\Mbarr_y(\Xi)$ and $M_-(\Xi)$ given in \ref{eqn:cbarr_definition}, \ref{eqn:Mbarr_y_definition} and \ref{eqn:M_-_definition}, respectively. 
\section{Asymptotics of the Murmuration Function}\label{section:asymptotics}
Throughout this section, $\Phi$ denotes a compactly supported smooth weight function. Recall that the corresponding murmuration function may be obtained from the murmuration density via $$M_\Phi(\Xi)=\frac{\int_0^\infty M(\Xi/u)\Phi(u)u^{\delta}\frac{du}{u}}{\int_0^\infty\Phi(u)u^\delta\frac{du}{u}}$$ as in \ref{eqn:murmuration_function_and_density_relation}. We show that for all such weight function $\Phi$ we obtain $M_\Phi(\Xi)\to-\frac{1}{2}$ as $\Xi\to\infty$, which is consistent with \ref{eqn:1_level_density} and the 1-level density conjecture.

Recall the murmuration density \begin{equation}\label{eqn:murmuration_function_recall}
    M(\Xi)=\cbarr\sum_{1\leq y\leq2\sqrt{\Xi}}\overline{M}_y(\Xi)+M_-(\Xi)=\frac{11\zeta(2)\cbarr}{4A}\sum_{1\leq y\leq2\sqrt{\Xi}}\kappa(y)\sqrt{\frac{\Xi}{4\Xi-y^2}}-\frac{11\pi}{12A}\sqrt{\Xi},
\end{equation} where \begin{align*}
    \cbarr&=\frac{1}{3}\prod_{\ell>2}\left(1-\ell^{-2}-\frac{\ell^{-3}}{1-\ell^{-2}}\right),\\
    A&=\prod_p\left(1+\frac{p}{(p+1)^2(p-1)}\right),\\
    \kappa(y)&=2^{\ONE\{2\mid y\}}\prod_{q\mid y\text{ odd}}\left(1+\frac{q^2}{q^4-2q^2-q+1}\right).
\end{align*} 
First we use Poisson summation to rewrite the murmuration density more cleanly.
\begin{prop}
We have \begin{align*}
    M(\Xi)&=\frac{11\pi\zeta(2)\cbarr}{16A}\sqrt{\Xi}\sum_{d=1}^\infty \frac{Q(d)}{d}\sum_{m=1}^\infty J_0\left(\frac{4\pi m\sqrt{\Xi}}{d}\right)-\frac{1}{2},
\end{align*} where $$Q(d)=\mu^2(d)\prod_{q\mid d\text{ odd}}\frac{q^2}{q^4-2q^2-q+1}$$ and $J_0$ is the Bessel function.
\end{prop}
\begin{proof}
    Note $\kappa(y)=\sum_{d\mid y}Q(d)$. Define $$f_c(t):=\sqrt{\frac{1}{1-c^2t^2}\,}\chi_{[-1/c,1/c]},$$ so \begin{equation}\label{eqn:before_poisson_sum}
    \sum_{1\leq y\leq 2\sqrt{\Xi}}\kappa(y)\sqrt{\frac{\Xi}{4\Xi-y^2}}=\frac{1}{2}{\sum_{d=1}^\infty}Q(d)\sum_{1\leq t\leq\frac{2\sqrt{\Xi}}{d}}\sqrt{\frac{1}{1-d^2t^2/4\Xi}}=\frac{1}{2}\sum_{d=1}^\infty  Q(d)\sum_{t=1}^\infty f_{\frac{d}{2\sqrt{\Xi}}}(t).
\end{equation} The Fourier transform of $f_c(t)$ is $\widehat{f_c}(s)=\frac{\pi}{c}J_0\left(\frac{2\pi s}{c}\right),$ and Poisson summation gives $$2\sum_{t=1}^\infty f_{\frac{d}{2\sqrt{\Xi}}}(t)+1=2\sum_{m=1}^\infty\frac{\pi\sqrt{\Xi}}{d}J_0\left(\frac{4\pi m\sqrt{\Xi}}{d}\right)+\frac{\pi\sqrt{\Xi}}{d}.$$ Solving for the sum over $t$ and plugging this into \ref{eqn:before_poisson_sum} and \ref{eqn:murmuration_function_recall}, we get \begin{align*}
    M(\Xi)&=\frac{11\zeta(2)\cbarr}{4A}\left[\frac{\pi\sqrt{\Xi}}{4}\sum_{d=1}^\infty \frac{Q(d)}{d}\sum_{m=1}^\infty J_0\left(\frac{4\pi m\sqrt{\Xi}}{d}\right)-\frac{1}{4}\sum_{d=1}^\infty Q(d)+\frac{\pi\sqrt{\Xi}}{2}\sum_{d=1}^\infty\frac{Q(d)}{d}\right]-\frac{11\pi}{12A}\sqrt{\Xi}.
\end{align*} Now the Euler products for $\zeta(2),\cbarr,A$ and $Q(d)$ give \begin{equation*}
    \frac{\zeta(2)\cbarr}{A}\sum_{d=1}^\infty Q(d)=\frac{8}{11},\qquad\zeta(2)\cbarr\sum_{d=1}^\infty \frac{Q(d)}{d}=\frac{2}{3},
\end{equation*} which yields the Proposition.
\end{proof}


\begin{thm}
For any compactly supported smooth weight function $\Phi:(0,\infty)\to[0,\infty)$, we have $$M_\Phi(\Xi):=\frac{\int_0^\infty M(\Xi/u)\Phi(u)u^{3/2}\frac{du}{u}}{\int_0^\infty\Phi(u)u^{3/2}\frac{du}{u}}\to-\frac{1}{2}\qquad\text{as $\Xi\to\infty$}.$$
\end{thm}
\begin{proof}
First we define $M'(\Xi):=M(\Xi^2)$ and $\Phi'(v):=2v^{3/2}\Phi(v^2).$ A simple change of variables shows that $M'_\Phi(\Xi^2)=M_{\Phi'}'(\Xi)$, so it suffices to prove the statement for the function $M'.$


Setting $x_m=4\pi m\Xi$, we have \begin{align*}
    M'(\Xi/u)+\frac{1}{2}&=\frac{11\pi\zeta(2)\cbarr}{16A}\frac{\Xi}{u}\sum_{d=1}^\infty \frac{Q(d)}{d}\sum_{m=1}^\infty J_0\left(\frac{4\pi m\Xi}{du}\right)\\
    &=\frac{11\zeta(2)\cbarr}{64A}\sum_{d,m}\frac{1}{m}\frac{Q(d)}{d}\frac{x_m}{u}J_0\left(\frac{x_m}{du}\right)\\
    &=\frac{11\zeta(2)\cbarr}{64A}\sum_{m=1}^\infty\frac{1}{m}F\left(\frac{x_m}{u}\right)
\end{align*} where $$F(x):=\sum_{d=1}^\infty\frac{Q(d)}{d}xJ_0\left(\frac{x}{d}\right)=\sum_{d=1}^\infty Q(d)f\left(\frac{d}{x}\right)$$ and $f(x):=x^{-1}J_0(1/x)$ is the same as that on \cite{Zubrilina}, p.23. The Mellin transform of $f$ is given by $$\tilde{f}(s)=\frac{1}{2^s}\frac{\Gamma\left(\frac{1-s}{2}\right)}{\Gamma\left(\frac{1+s}{2}\right)}$$ on the strip $-\frac{1}{2}<\Re(s)<1$. Mellin inversion yields $$F(x)=\sum_{d=1}^\infty Q(d)\frac{1}{2\pi i}\int_{\left(1/2\right)}\ftilde(s)\left(\frac{d}{x}\right)^{-s}\d s=\frac{1}{2\pi i}\int_{(1/2)}L(s)\ftilde(s)x^s\d x$$ where \begin{align*}
    L(s)&:=\sum_{d=1}^\infty Q(d)d^{-s}=\frac{\zeta(s+2)}{\zeta(2s+4)}\frac{1+2^{-s}}{1+\frac{4}{7}\cdot2^{-s}}\prod_{p>2}\left(1+\frac{2p^2+p-1}{(p^4-2p^2-p+1)(p^{s+2}+1)}\right).
\end{align*}\footnote{This differs from that of \cite{Zubrilina} by a factor due to the omission of the factor at 2 in our definition of $Q(d)$.} Thus $$\tPhi\left(\tfrac{3}{2}\right)F_\Phi(x)=\frac{1}{2\pi i}\int_{(1/2)}L(s)\ftilde(s)\tPhi\left(\tfrac{3}{2}-s\right)x^s\d s$$ Since $\Phi$ is smooth and compactly supported, $\tPhi$ is entire and has rapid decay with respect to $t=\Im(s)$. The $L$-function converges uniformly on $\Re(s)>-\frac{1}{2}-\epsilon$ for any small $\epsilon>0$, and we can shift the contour to the line $\Re(s)=-\frac{1}{2}$ picking up no residues, giving $$F_\Phi(x)\ll\int_{(-1/2)}L(s)\ftilde(s)\tPhi\left(\tfrac{3}{2}-s\right)x^s\d s\ll x^{-1/2}.$$ Therefore, $$M_\Phi'(\Xi/u)+\frac{1}{2}=\frac{11\zeta(2)\cbarr}{64A}\sum_{m=1}^\infty\frac{1}{m}F_\Phi(x_m)\ll\sum_{m=1}^\infty\frac{1}{m}(m\Xi)^{-1/2}\ll\Xi^{-1/2}.$$
\end{proof}

\renewcommand{\thesection}{\Alph{section}}
\renewcommand{\theequation}{(\Alph{section}.\arabic{equation})}
\section*{Appendix}

\setcounter{section}{1}
\setcounter{subsection}{0}
\setcounter{equation}{0}

\addcontentsline{toc}{section}{Appendix}

\counterwithin{equation}{section}
\renewcommand{\theequation}{(\Alph{section}.\arabic{equation})}

\renewcommand{\thesubsection}{\Alph{section}.\arabic{subsection}}
\renewcommand{\thesubsubsection}{\Alph{section}.\arabic{subsection}.\arabic{subsubsection}}

\subsection{Lemmas}\label{A:lemmas}
\subsubsection{Truncation of $L(1,\chi_D)$}
\counterwithout{thm}{subsection}
\counterwithin{thm}{section}

Recall that $\chi_{D}=\Legendre{D}{\cdot}$ is the Kronecker symbol. The following estimate is based on Bombieri's density theorem (see \cite{Bombieri}, p.204, Theorem 5).
\setcounter{thm}{0}
\begin{lem}\label{lem:density_theorem}
Let $\msD$ be a finite set of fundamental discriminants. Let $T$ be an arbitrary cutoff parameter and $\frac{1}{2}\leq\sigma\leq1$. Then $$\sum_{D\in\msD}\abs{L(1,\chi_D)-\sum_{n\leq T}\frac{\chi_D(n)}{n}}\ll T^{-1}UX^{\frac{8(1-\sigma)}{3-2\sigma}+\frac{1}{2}+\epsilon}+\abs{\msD}\left(T^{\sigma-1}+\frac{1}{U}\right)(TX)^\epsilon,$$ where $X=\max_{D\in\msD}\abs{D}$ and $U\ll X$ is another cutoff parameter.
\end{lem}
\begin{proof}
Choose $U\geq2$ and near a half integer. Let $\chi$ be a character modulo $q=\abs{D}$ and $N(\sigma,U;\chi)$ be the number of zeroes of the Dirichlet $L$-function $L(s,\chi)$ in the box $\sigma\leq\Re(s)\leq1,\abs{\Im(s)}\leq U$. Then the abovementioned density theorem implies that \begin{align}
    \sum_{D\in\msD}N(\sigma,U;\chi_D)&\leq\sum_{D\in\msD}\frac{1}{\phi(q)}\sum_{\text{$\chi$ mod $q$}}\abs{\tau(\chi)}^2N(\sigma,U;\chi)\nonumber\\
    &\ll U(X^2+UX)^{4(1-\sigma)/(3-2\sigma)}(UX)^\epsilon\nonumber\\
    &\ll UX^{8(1-\sigma)/(3-2\sigma)+\epsilon},\label{eqn:density_theorem}
\end{align} where $\tau(\chi)=\sum_{a=1}^{q}\chi(a)e(a/q)$ is the Gauss sum and $\abs{\tau(\chi_D)}^2=q$ since $\chi_D$ is primitive (see \cite{Davenport}, pp.66-67). Let $M=UX^{8(1-\sigma)/(3-2\sigma)+\epsilon}$ denote the bound in \ref{eqn:density_theorem}, and this is the upper bound for the number of those $D$ with $N(\sigma,U;\chi_D)>0$. For any such $D$, partial summation and P\'{o}lya-Vinogradov gives (see \cite{Ayoub}, p.321, Theorem 5.2): \begin{equation}\label{eqn:Polya_Vinogradov_bound_for_exceptional_L}
    L(1,\chi_D)-\sum_{n\leq T}\frac{\chi_D(n)}{n}\ll T^{-1}\abs{D}^{1/2+\epsilon}.
\end{equation} 

For all $D$ such that $N(\sigma,U;\chi_D)=0$, Perron's Formula (see \cite{Titchmarsh}, p.301) implies $$\sum_{n\leq T}\frac{\chi_D(n)}{n}=\frac{1}{2\pi i}\int_{\epsilon-iU}^{\epsilon+iU}L(1+s,\chi_D)T^s\frac{ds}{s}+O\left(\frac{T^\epsilon}{U}\right).$$ Shifting the contour of integration to the line $\Re(s)=\sigma-1+\epsilon$ gives \begin{equation}\label{eqn:Lemma_A2_after_shifting_contour}
    \sum_{n\leq T}\frac{\chi_D(n)}{n}=L(1,\chi_D)+\frac{1}{2\pi i}\int_{\sigma-1+\epsilon-iU}^{\sigma-1+\epsilon+iU}L(1+s,\chi_D)T^s\frac{ds}{s}+O\left(\frac{T^\epsilon}{U}\right).
\end{equation} Indeed, since $N(\sigma,U;\chi_D)=0$ the integrand has only a simple pole at $s=0$ with residue $L(1,\chi_D).$ Now we need an estimate on $L(1+s,\chi_D)$ along $\Re(s)=\sigma-1+\epsilon$. Following the proof of the Lindel\"{o}f Hypothesis from GRH in \cite{Iwaniec-Kowalski}, pp.114-116, we obtain (in the following, $\mfq:=\mfq(\chi_D,s)=\abs{D}(\abs{s}+3)$ is the ``analytic conductor'') $$-\frac{L'}{L}(s,\chi_D)\ll\frac{\log\mfq}{\sigma-\sigma}X^{\sigma-\sigma}+X^{1-\sigma}+\log X\ll\frac{(\log\mfq)^{(1-\sigma)/(1-\sigma)}}{\sigma-\sigma}$$ for any $s$ with $\sigma<\sigma\leq\frac{5}{4}$, and we chose $X=(\log\mfq)^{1/(1-\sigma)}$. This further implies $$\log L(s,\chi_D)\ll\frac{(\log\mfq)^{(1-\sigma)/(1-\sigma)}}{(\sigma-\sigma)(\log\log\mfq)}+\log\log\mfq,$$ and thus $$L(s,\chi_D)\ll\mfq^\epsilon\ll(\abs{Ds})^\epsilon$$ for $\sigma\geq\sigma+\epsilon$ (the implied constant depend on $\epsilon$ only). Putting this bound into \ref{eqn:Lemma_A2_after_shifting_contour} gives \begin{equation}\label{eqn:bound_for_L_where_GRH_holds}
    \sum_{n\leq T}\frac{\chi_D(n)}{n}-L(1,\chi_D)\ll(UX)^\epsilon T^{\sigma-1}+\frac{T^\epsilon}{U}.
\end{equation}

Finally, using \ref{eqn:Polya_Vinogradov_bound_for_exceptional_L} for those $D\in\msD$ with $N(\sigma,U;\chi_D)>0$ and \ref{eqn:bound_for_L_where_GRH_holds} for the rest, we obtain \begin{align*}
    \sum_{D\in\msD}\abs{L(1,\chi_D)-\sum_{n\leq T}\frac{\chi_D(n)}{n}}&\ll MT^{-1}X^{1/2+\epsilon}+\abs{\msD}\left((UX)^\epsilon T^{\sigma-1}+\frac{T^\epsilon}{U}\right)\\
    &\ll T^{-1}UX^{\frac{8(1-\sigma)}{3-2\sigma}+\frac{1}{2}+\epsilon}+\abs{\msD}\left(T^{\sigma-1}+\frac{1}{U}\right)(TX)^\epsilon,
\end{align*} as desired.
\end{proof}

\begin{cor}\label{cor:truncation_assuming_GRH}
Let $\msD$ and $X$ be as above. Assuming GRH (specifically the Riemann Hypothesis for Dirichlet $L$-functions), we have $$\sum_{D\in\msD}\abs{L(1,\chi_D)-\sum_{n\leq T}\frac{\chi_D(n)}{n}}\ll\abs{\msD}T^{-1/2}(TX)^\epsilon.$$
\end{cor}
\begin{proof}
We take $\sigma=\frac{1}{2}+\epsilon$ and apply the previous Lemma. On GRH the first error term vanishes since it comes from exceptional $L$-functions with $N(\sigma,U;\chi_D)>0$, and taking $U$ a large power of $T$ (we also no longer need $U\ll X$) gives the bound.
\end{proof}

\subsubsection{Arithmetic Functions}
\begin{lem}[Lemma 6.7 of \cite{Zubrilina}, p.33]\label{lem:partial_sum_characters_over_squarefree_numbers}
Let $m\in\NN$, and let $\chi$ be a real character modulo $m$ coming from a primitive character modulo $m_0$. Let $$\sigma_\chi(N)=\sum_{n\leq N}\mu^2(n)\chi(n).$$ Then \begin{equation*}
    \sigma_{\chi}(N)=\begin{dcases}
        N\cdot\frac{\eta(m)}{\zeta(2)}+O\left(N^{3/5+\epsilon}m^\epsilon\right) &\text{if $\chi$ is the principal character,}\\
        O\left(N^{3/5+\epsilon}m^\epsilon m_0^{1/5+\epsilon}\right) &\text{otherwise}.
    \end{dcases}
\end{equation*} Here $\eta(m)=\prod_{p\mid m}\frac{p}{p+1}$.
\end{lem}

\begin{lem}[Lemma 6.5 of \cite{Zubrilina}, p.29]\label{lem:sum_of_eta(2m)}
We have $$\sum_{m\leq T}\frac{\eta(2m)}{m^2}=\frac{8A}{11}+O\left(T^{-1}\right),$$ where $$A=\prod_p\left(1+\frac{p}{(p+1)^2(p-1)}\right)$$ as in \ref{eqn:A_definition}.
\end{lem}

\subsubsection{Miscellaneous}
\begin{lem}\label{lem:sum_of_(x^2-a/p)_over_residue_class}
Let $p$ be an odd prime and $a$ an integer with $p\nmid a$. Then $$\sum_{x\,\mathrm{mod}\,p}\Legendre{x^2-a}{p}=-1.$$
\end{lem}
\begin{proof}
Note that $\sum_x\Legendre{x^2-a}{p}=0$ only depends on square class of $a$ mod $p$ by change of variables, and \begin{align*}
    \sum_x\left(\frac{x^2-1}{p}\right)&=\sum_{x\neq 1}\Legendre{(x+1)(x-1)^{-1}}{p}\\
    &=\sum_{y\neq 0}\Legendre{1+2y}{p} &&(x=y^{-1}+1)\\
    &=\sum_{z\neq1}\Legendre{z}{p} &&(y=2^{-1}(z-1))\\
    &=-1.
\end{align*} Thus $\sum_x\Legendre{x^2-\square}{p}=-1$. Moreover, \begin{align*}
    0&=\sum_x\sum_a\Legendre{x^2-a}{p}\\
    &=\sum_a\sum_x\Legendre{x^2-a}{p}\\
    &=\sum_{a=0}\sum_x\Legendre{x^2-a}{p}+\sum_{a=\square}\sum_x\Legendre{x^2-a}{p}+\sum_{a=\notsquare}\sum_x\Legendre{x^2-a}{p}\\
    &=(p-1)+\frac{p-1}{2}(-1)+\frac{p-1}{2}\sum_x\Legendre{x^2-\notsquare}{p},
\end{align*} so $\sum_x\Legendre{x^2-\notsquare}{p}=-1$ as well.
\end{proof}

\subsection{Numerics}\label{B:numerics}
Using \ref{eqn:murmuration_function_and_density_relation}, one can integrate the murmuration density $M(\Xi)$ given in \ref{eqn:murmuration_density} to obtain murmuration functions $M_\Phi(\Xi)$ for any weight function $\Phi$: $$M_\Phi(\Xi)=\frac{\int_0^\infty M(\Xi/u)\Phi(u)u^{1/2}\d u}{\int_0^\infty\Phi(u)u^{1/2}\d u}.$$ Computer numerics can then be conducted to verify the correctness of our computations. This is most easily done for the characteristic function $\Phi=\chi_{[1,2]}$, for which we numerically compute the average $$G(p,N)=\frac{\sum_{\substack{f\in\msF\\N_f\in[N,2N]}}a_f(p)\sqrt{p}}{\sum_{\substack{f\in\msF\\N_f\in[N,2N]}}1},$$ as well as with a second average over $p$. We present the resulting plots in this section.

\subsubsection{Plots of the Vertical Average $G(p,N)$ for Large $N$}
First we simply plot $G(p,N)$ as a function of $\xi=p/N$ for large $N$ (here $N=2^{19}$) up to some range. As one can see in the following images, there is arithmetic dependence of $G(p,N)$ on $p$, both in the sense of $\delta_y(p)$ and of $c(p)$, which appear in our analysis.

\begin{figure}[H]
    \centering
    \begin{subfigure}{\textwidth}
        \includegraphics[width=\textwidth]{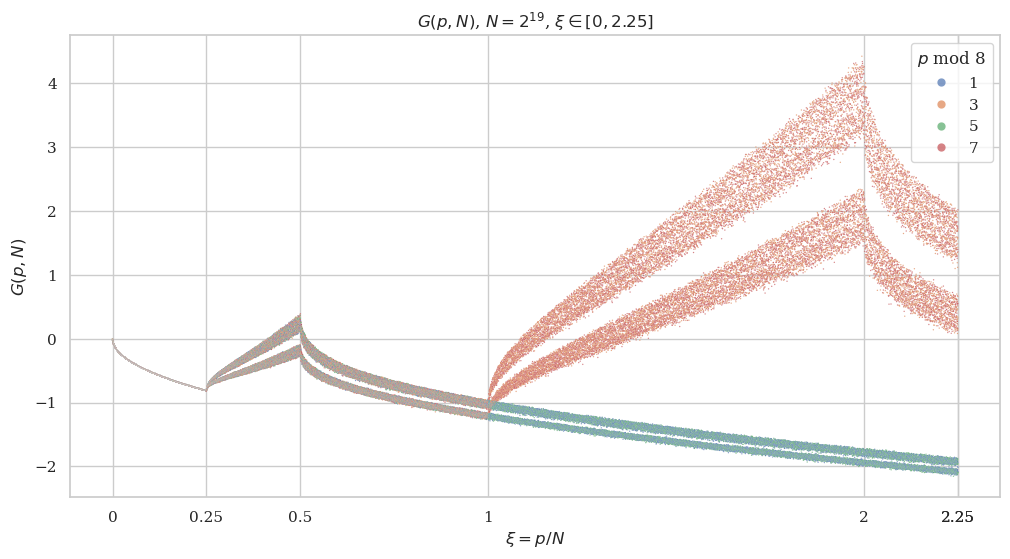}
        \caption{Primes colored modulo 8. One can see that the 3-mod-4 primes have $G_{\num,2}^+$ contribution while the 1-mod-4 primes do not; this corresponds to $\delta_2(p)=\ONE\{p\equiv3\Mod{4}\}$ in our analysis.}
    \end{subfigure}
\end{figure}

\begin{figure}[H]
    \centering
    \begin{subfigure}{\textwidth}
        \includegraphics[width=\textwidth]{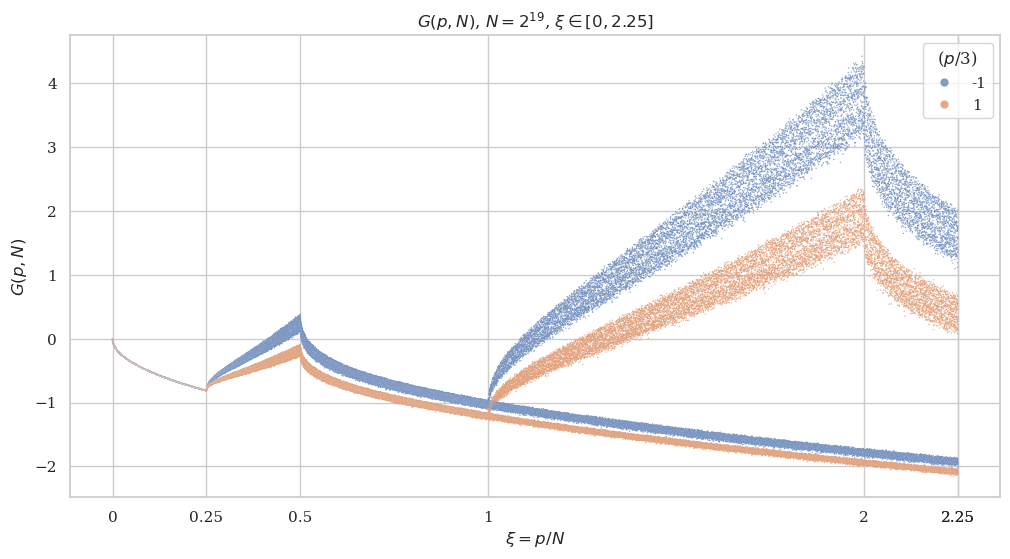}
        \setcounter{subfigure}{1}
        \caption{Primes colored modulo 3. One can see that $G_{\num}^+$ is larger for the non-squares modulo 3; this corresponds to the local factor $1-2\ell^{-3}-\frac{2\ell^{-3}}{1-\ell^{-2}}$ of $c(p)$ at $\ell=3$.}
    \end{subfigure}
    \begin{subfigure}{\textwidth}
        \includegraphics[width=\textwidth]{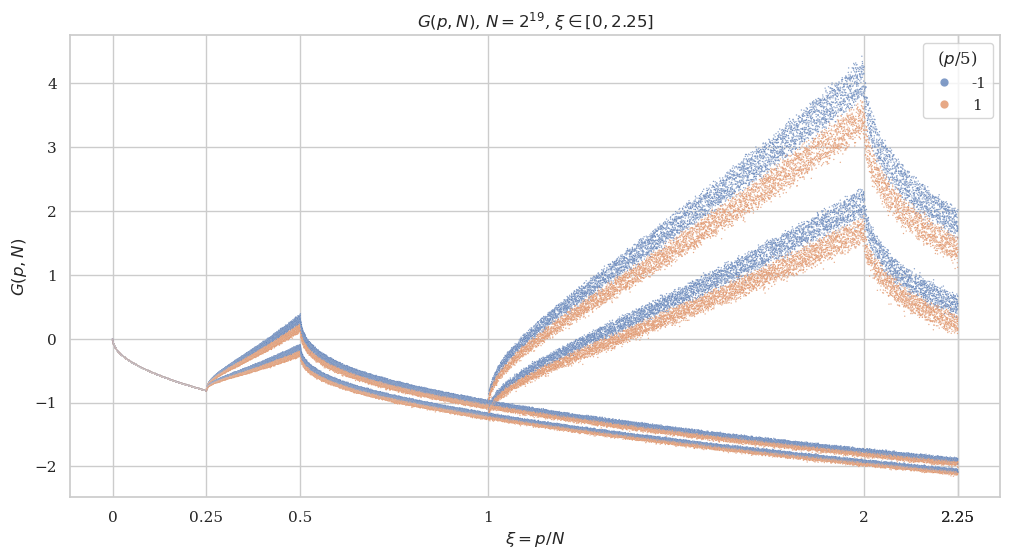}
        \setcounter{subfigure}{2}
        \caption{Primes colored modulo 5. One can further see the effect of the local factor $1-2\ell^{-3}-\frac{2\ell^{-3}}{1-\ell^{-2}}$ of $c(p)$ at $\ell=5$.}
    \end{subfigure}
    \setcounter{figure}{2}
    \caption{$G(p,N)$ vs $\xi=p/N$ up to $\xi=9/4$ for $N=2^{19}$, with $p$ colored according to different moduli.}
\end{figure}

\begin{figure}[H]
    \centering
    \begin{subfigure}{\textwidth}
        \includegraphics[width=\textwidth]{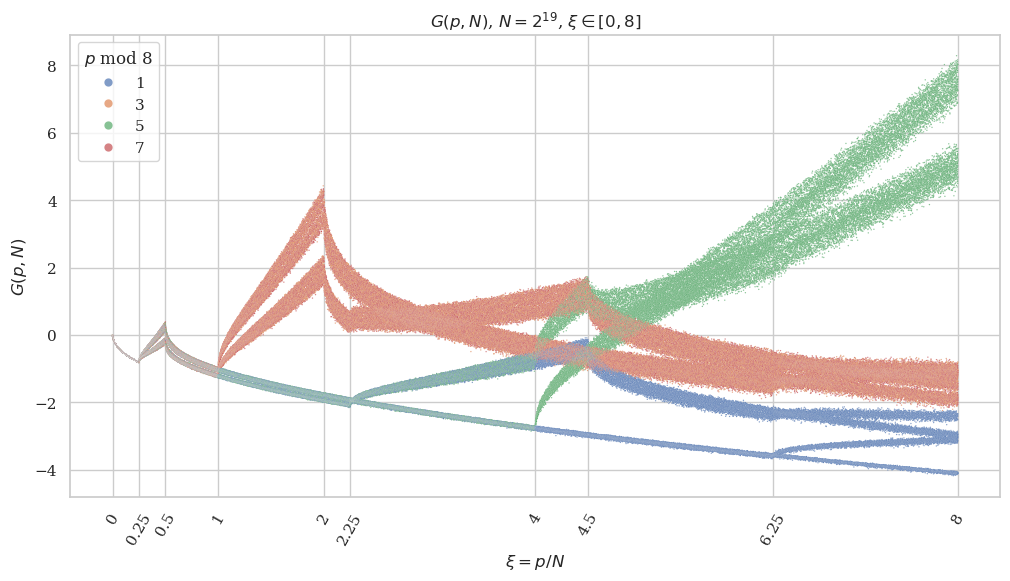}
        \setcounter{subfigure}{0}
        \caption{Primes colored modulo 8. One can further see the effect of $\delta_4(y)=\ONE\{p\equiv5\Mod{8}\}$ for the term entering at $\xi=4$.}
    \end{subfigure}
    \begin{subfigure}{\textwidth}
        \includegraphics[width=\textwidth]{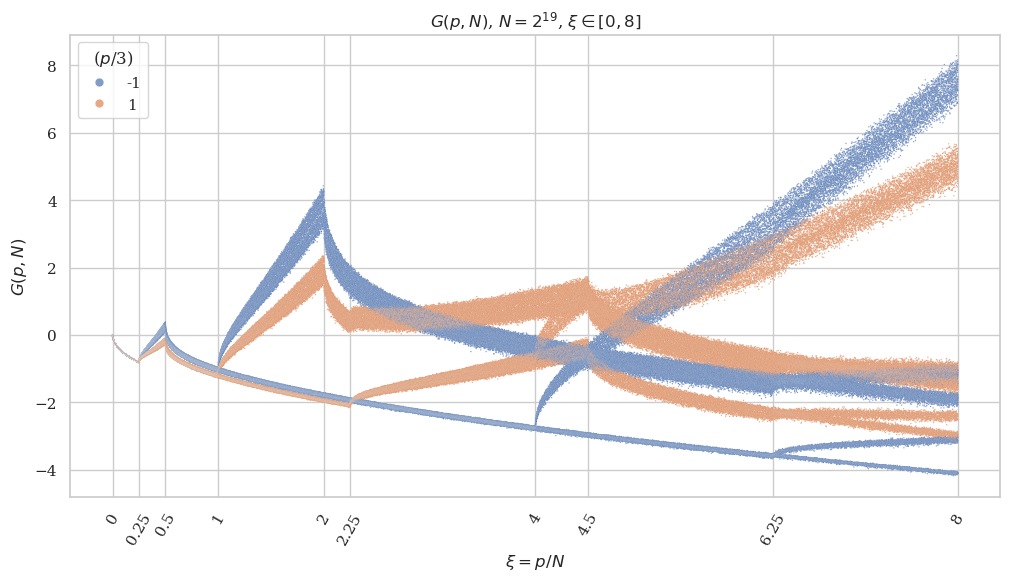}
        \setcounter{subfigure}{1}
        \caption{Primes colored modulo 3. One can further see the effect of $\delta_3(y)=\ONE\{(p/3)=1\}$ for the term entering at $\xi=9/4$.}
    \end{subfigure}
\end{figure}

\begin{figure}[H]
    \centering
    \begin{subfigure}{\textwidth}
        \includegraphics[width=\textwidth]{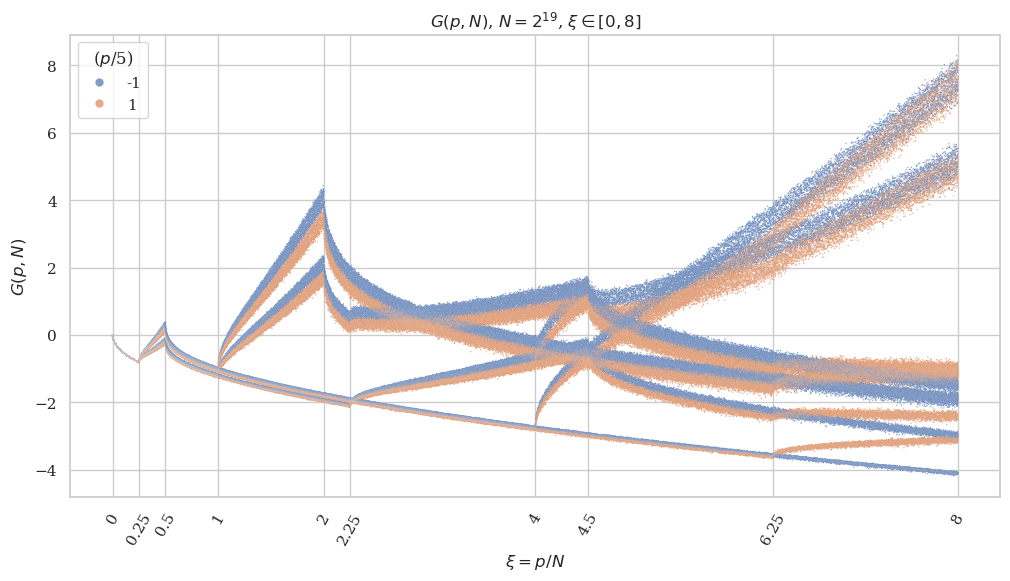}
        \setcounter{subfigure}{2}
        \caption{Primes colored modulo 5. One can further see the effect of $\delta_5(y)=\ONE\{(p/5)=1\}$ for the term entering at $\xi=25/4$.}
    \end{subfigure}
    \setcounter{figure}{3}
    \caption{$G(p,N)$ vs $\xi=p/N$ up to $\xi=8$ for $N=2^{19}$, with $p$ colored according to different moduli.}
\end{figure}

\subsubsection{Plots of the Vertical Average $G(p,N)$ for Different $N$}
Modulo the arithmetic independence of $G(p,N)$ on $p$ (no pun intended), it is a function of $\xi=p/N$ only. That is, we should observe ``scale invariance'' as we vary $N$. Indeed this phenomenon is very pronounced, as shown in the next few pages.

\begin{figure}[H]
    \centering
    \begin{subfigure}{\textwidth}
        \includegraphics[width=\textwidth]{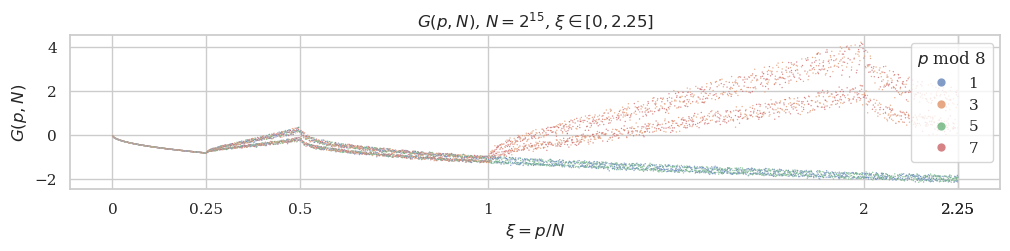}
    \end{subfigure}
    \begin{subfigure}{\textwidth}
        \includegraphics[width=\textwidth]{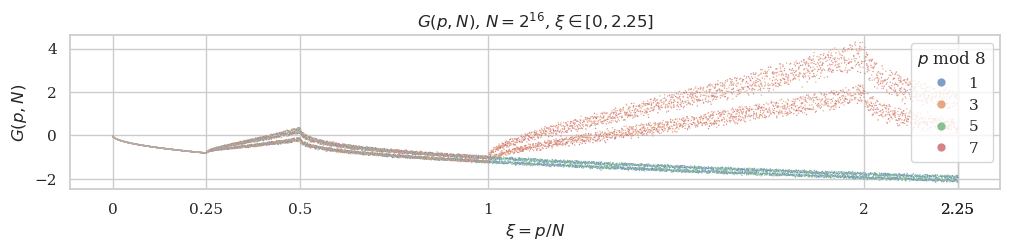}
    \end{subfigure}
    \begin{subfigure}{\textwidth}
        \includegraphics[width=\textwidth]{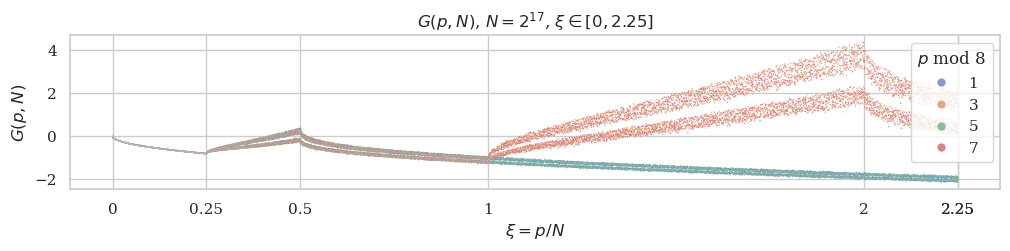}
    \end{subfigure}
    \begin{subfigure}{\textwidth}
        \includegraphics[width=\textwidth]{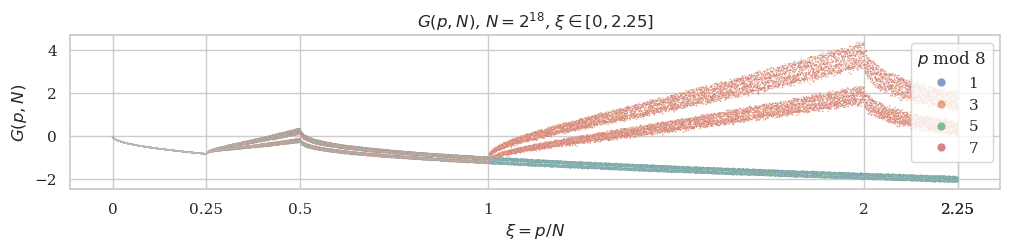}
    \end{subfigure}
    \begin{subfigure}{\textwidth}
        \includegraphics[width=\textwidth]{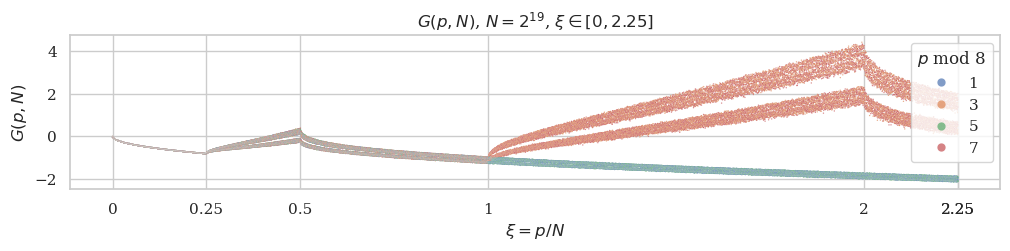}
    \end{subfigure}
    \caption{$G(p,N)$ vs $\xi=p/N$ up to $\xi=9/4$ for $N=2^k$ and $15\leq k\leq 19$, with $p$ colored modulo 8.}
\end{figure}

\begin{figure}[H]
    \centering
    \begin{subfigure}{\textwidth}
        \includegraphics[width=\textwidth]{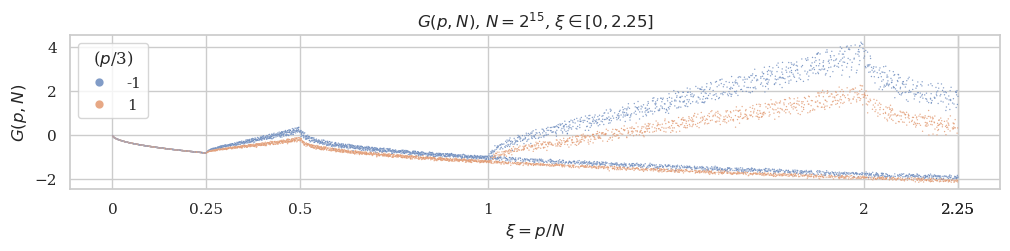}
    \end{subfigure}
    \begin{subfigure}{\textwidth}
        \includegraphics[width=\textwidth]{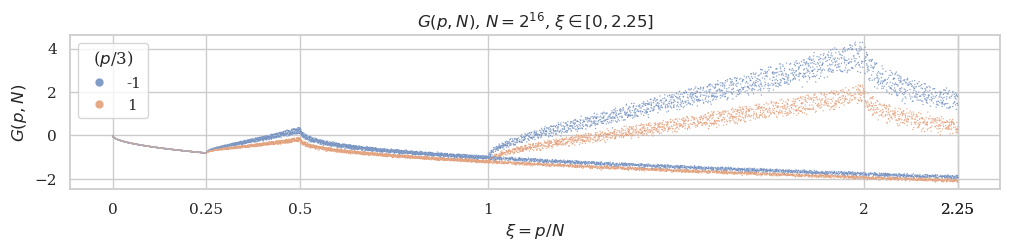}
    \end{subfigure}
    \begin{subfigure}{\textwidth}
        \includegraphics[width=\textwidth]{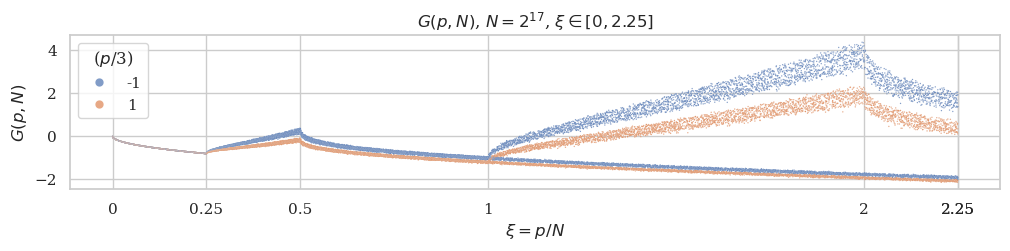}
    \end{subfigure}
    \begin{subfigure}{\textwidth}
        \includegraphics[width=\textwidth]{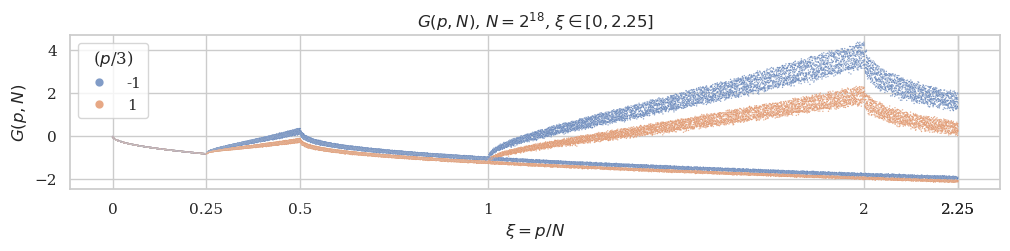}
    \end{subfigure}
    \begin{subfigure}{\textwidth}
        \includegraphics[width=\textwidth]{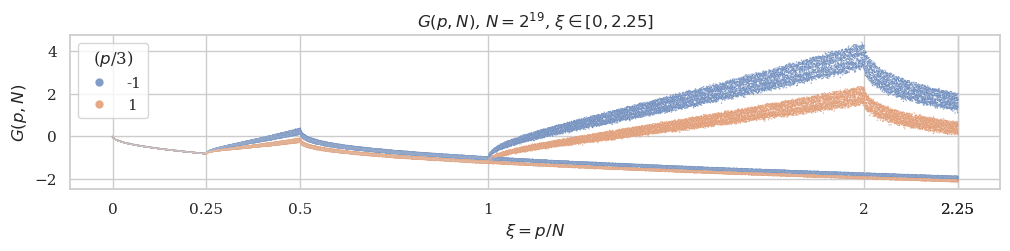}
    \end{subfigure}
    \caption{$G(p,N)$ vs $\xi=p/N$ up to $\xi=9/4$ for $N=2^k$ and $15\leq k\leq 19$, with $p$ colored modulo 3.}
\end{figure}

\subsubsection{Averaging over $p$}
Now we plot $$G_{\avg}(P,N)=\Exp_{p\in[P,P+H]}G(p,N)$$ with $H=P^{0.51}$ (in our analysis we require $H\gg P^{1/2+\delta}$ for arbitrary $\delta>0$ \TODO). We expect this to match the murmuration function $M_\Phi(\Xi)$ obtained from integrating the murmuration density, which is indeed verified by numerics. 



\begin{figure}[H]
    \centering
    \begin{subfigure}{0.95\textwidth}
        \includegraphics[width=\textwidth]{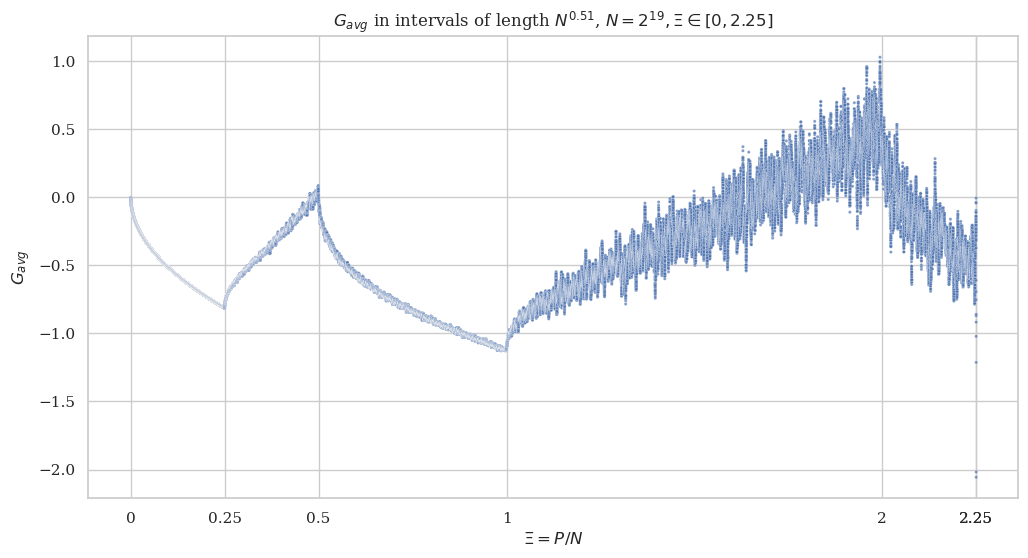}
        \caption{$G_{\avg}(P,N)$ only.}
    \end{subfigure}
    \begin{subfigure}{0.95\textwidth}
        \includegraphics[width=\textwidth]{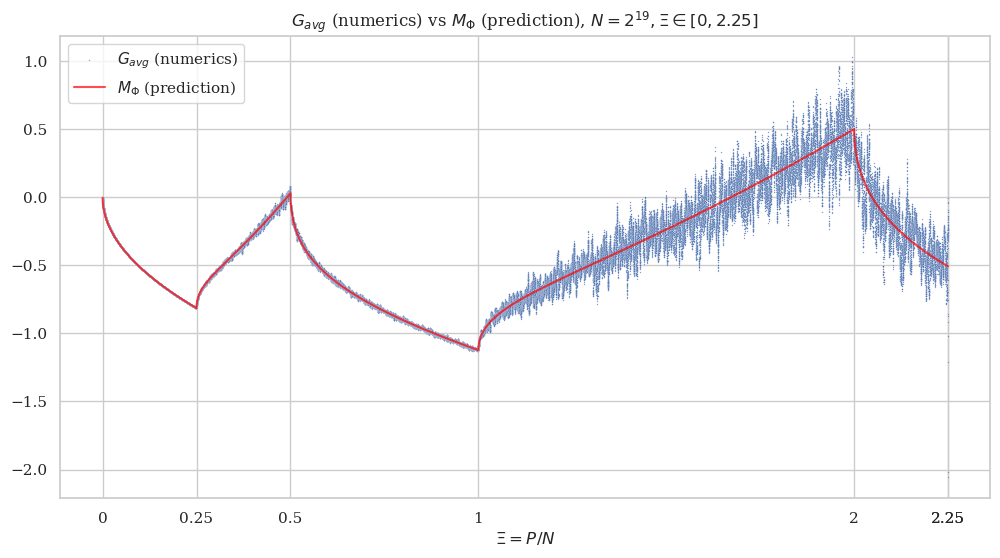}
        \caption{$G_{\avg}(P,N)$ with $M_\Phi(\Xi)$}
    \end{subfigure}
    \caption{$G_{\avg}(P,N)$ vs $\Xi=P/N$ up to $\Xi=9/4$ for $N=2^{19}$.}
\end{figure}

\begin{figure}[htbp]
    \centering
    \begin{subfigure}{\textwidth}
        \includegraphics[width=\textwidth]{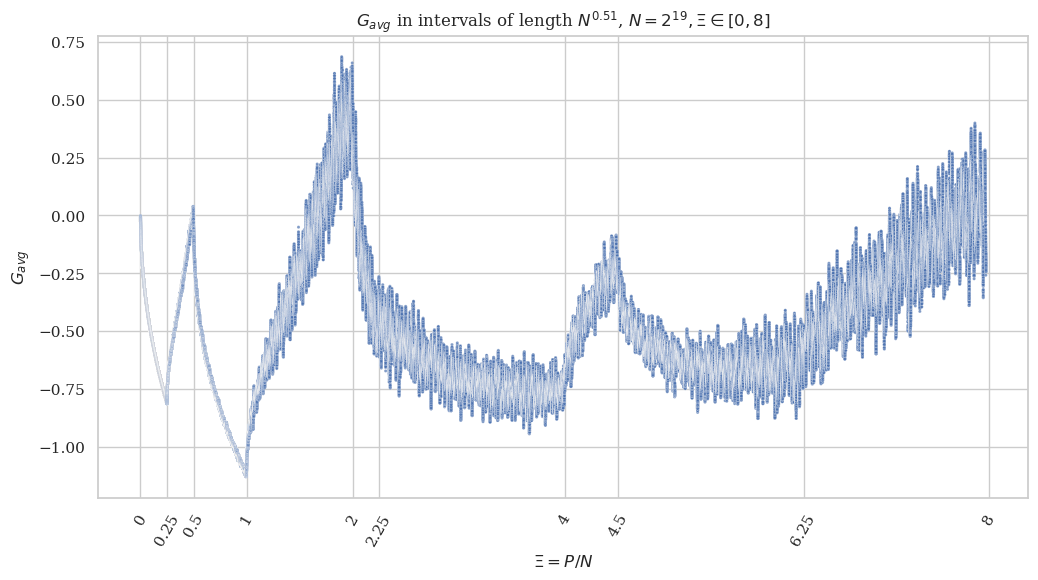}
        \caption{$G_{\avg}(P,N)$ only.}
    \end{subfigure}
    \begin{subfigure}{\textwidth}
        \includegraphics[width=\textwidth]{pics/murmuration_function/murmuration_function,gamma_to_8_with_pred.png}
        \caption{$G_{\avg}(P,N)$ with $M_\Phi(\Xi)$}
    \end{subfigure}
    \caption{$G_{\avg}(P,N)$ vs $\Xi=P/N$ up to $\Xi=8$ for $N=2^{19}$.}
\end{figure}

\clearpage
\subsubsection{Asymptotic Behavior of $M_\Phi(\Xi)$}
Plotting $M_\Phi(\Xi)$ shows that it converges to $-\frac{1}{2}$ as $\Xi\to\infty$, consistent with the 1-level density conjecture.

\begin{figure}[H]
    \centering
    \includegraphics[width=\textwidth]{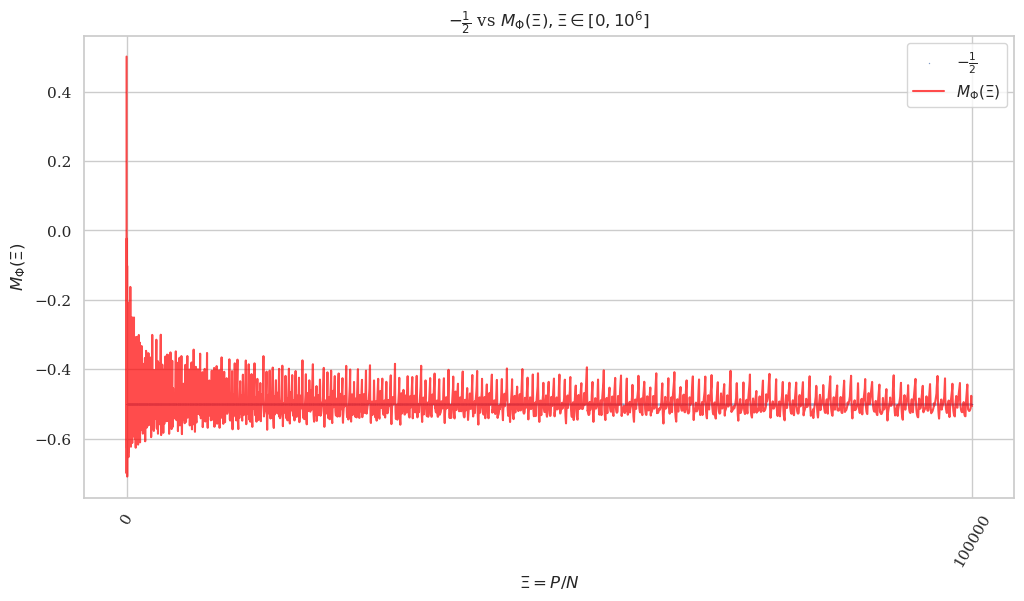}
    \caption{$M_\Phi(\Xi)$ for $\Xi=P/N$ up to $10^6$.}
\end{figure}

\clearpage
\printbibliography

\end{document}